\documentclass[12pt, reqno]{article}
\title{
\centerline{
\begin{minipage}{1.35\textwidth}
\centering
{\Large
Indistinguishability of Trees in Uniform Spanning Forests
}
\end{minipage}
}
}
\author{Tom Hutchcroft and Asaf Nachmias}

\AtEndDocument{\bigskip{\footnotesize%
\par
\textsc{Tom Hutchcroft} \par
\textsc{Department of Mathematics, University of British Columbia} \par
  \textit{Email:} \texttt{thutch@math.ubc.ca} \par
  \addvspace{\medskipamount}
\textsc{Asaf Nachmias} \par
\textsc{Department of Mathematics, University of British Columbia and} \par \textsc{School of Mathematical Sciences, Tel Aviv University} \par
\textit{Email:} \texttt{asafnach@post.tau.ac.il} \par
}}

\usepackage{amsmath,amssymb,amsfonts,amsthm}
\usepackage{color}

\usepackage[colorlinks=true,linkcolor=blue,citecolor=black,urlcolor=black,pdfborder={0 0 0}]{hyperref}
\usepackage{graphicx}
\usepackage{cleveref}
\usepackage{esvect}
\usepackage{soul}

\crefname{theorem}{Theorem}{Theorems}
\crefname{thm}{Theorem}{Theorems}
\crefname{lemma}{Lemma}{Lemmas}
\crefname{lem}{Lemma}{Lemmas}
\crefname{remark}{Remark}{Remarks}
\crefname{prop}{Proposition}{Propositions}
\crefname{defn}{Definition}{Definitions}
\crefname{corollary}{Corollary}{Corollaries}
\crefname{conjecture}{Conjecture}{Conjectures}
\crefname{question}{Question}{Questions}
\crefname{chapter}{Chapter}{Chapters}
\crefname{section}{Section}{Sections}
\crefname{figure}{Figure}{Figures}
\newcommand{\crefrangeconjunction}{--}

\theoremstyle{plain}
\newtheorem{thm}{Theorem}[section]
\newtheorem{lemma}[thm]{Lemma}
\newtheorem{lem}[thm]{Lemma}
\newtheorem{corollary}[thm]{Corollary}
\newtheorem{prop}[thm]{Proposition}

\newtheorem{question}[thm]{Question}
\newtheorem*{question*}{Question}
\theoremstyle{definition}
\newtheorem{defn}[thm]{Definition}
\newtheorem{example}[thm]{Example}

\theoremstyle{remark}
\newtheorem*{remark}{Remark}

\numberwithin{equation}{section}

\renewcommand{\P}{\mathbb P}
\newcommand{\Z}{\mathbb Z}
\newcommand{\E}{\mathbb E}

\newcommand{\eps}{\varepsilon}

\usepackage{color}

\newcommand{\URN}{unimodular random rooted network }
\newcommand{\URNs}{unimodular random rooted networks }

\newcommand{\ch}[1]{#1}
\newcommand{\ca}[1]{{#1}}

\newcommand{\asaf}[1]{{#1}}
\newcommand{\tom}[1]{{#1}}

\newcommand{\ct}[1]{{#1}}
\newcommand{\cha}[1]{#1}

\usepackage{enumitem}
 \usepackage[margin=1.3in]{geometry}

 \usepackage{bbm}
\usepackage{caption}
\usepackage{thmtools}
\usepackage{commath}
\usepackage{mathrsfs}
\usepackage{bbm}
\usepackage{relsize}
\usepackage{enumitem}
\usepackage{eufrak}

\newcommand{\symdif}{\hspace{.1em}\triangle\hspace{.1em}}

\newcommand{\eqd} {\overset{d}{=}}

\newcommand{\be}{\begin{equation}}
\newcommand{\ee}{\end{equation}}

\newcommand{\F}{\mathfrak F}

\newcommand{\A}{\mathscr A}
\newcommand{\B}{\mathscr B}

\newcommand{\cG}{\mathcal G}

\newcommand{\sA}{\mathscr A}
\newcommand{\sB}{\mathscr B}
\newcommand{\sC}{\mathscr C}
\newcommand{\sP}{\mathscr P}
\newcommand{\sE}{\mathscr E}
\newcommand{\sG}{\mathscr G}

\newcommand{\core}{\operatorname{core}}
\newcommand{\past}{\operatorname{past}}
\newcommand{\WUSF}{\mathsf{WUSF}}

\newcommand{\FUSF}{\mathsf{FUSF}}

\newcommand{\UST}{\mathsf{UST}}

\newcommand{\Gg}{\mathcal{G}_{\bullet \bullet}}
\newcommand{\Gb}{\mathcal{G}_{\bullet}^{\{0,1\}}}

\begin{document}
\maketitle
\begin{abstract}
We prove that in both the free and the wired uniform spanning forest \ch{(FUSF and WUSF)} of any unimodular random rooted network (in particular, of any Cayley graph), it is  impossible to distinguish the connected components of the forest from each other by invariantly defined \cha{graph} properties \ch{almost surely}. This confirms a conjecture of Benjamini, Lyons, Peres and Schramm~\cite{BLPS}.

 We \tom{also} answer positively two additional questions of \cite{BLPS} under the assumption of unimodularity. We prove that on any unimodular random rooted network, the FUSF is either connected or has infinitely many connected components almost surely, and, if the FUSF and WUSF are distinct, then every component of the FUSF is transient and infinitely-ended almost surely. All of these results are new even for Cayley graphs. \end{abstract}

\section{Introduction}

The \textbf{Free Uniform Spanning Forest} (FUSF) and the \textbf{Wired Uniform Spanning Forest} (WUSF) of an infinite graph $G$ are defined as weak limits of the uniform spanning trees on large finite subgraphs of $G$, taken with either free or wired boundary conditions  respectively (see \cref{S:background} for details).
First studied by Pemantle \cite{Pem91}, the USFs are closely related many other areas of probability, including electrical networks \cite{Kirch1847,BurPe93}, \ct{Lawler's} loop-erased random walk \cite{Lawler80,Wilson96,BLPS}, sampling algorithms \cite{ProppWilson,Wilson96}, domino tiling \cite{Ken00}, the Abelian sandpile model \cite{JarRed08,JarWer14,MajDhar92}, \ct{the rotor-router model \cite{HLMPPW}}, and the Fortuin-Kasteleyn random cluster model \cite{GrimFKbook,Hagg95}. The USFs are also of interest in group theory, where the FUSFs of Cayley graphs are related to the $\ell^2$-Betti numbers \cite{Gab05,Lyons09} and to the fixed price problem of Gaboriau \cite{Gab10}, and have also been used to approach the Dixmier problem \cite{EpMo09}.

Although both USFs  are defined as limits of trees, they need not be connected. Indeed, a principal result of Pemantle \cite{Pem91} is that the FUSF and WUSF coincide on $\Z^d$ for all $d\geq \ca{1}$ and that
they are  connected \ct{almost surely (a.s.)} if and only if $d \leq 4$.
 A complete characterisation of the connectivity of the WUSF was given by
 Benjamini, Lyons, Peres and Schramm (henceforth referred to as BLPS) in their seminal work \cite{BLPS}, who
showed that the WUSF of a graph $G$ is connected a.s.~if and only if the traces of two simple random walks started at arbitrary vertices of $G$ a.s.~intersect.
This recovers Pemantle's result on $\Z^d$, and shows more generally that the WUSF of a Cayley graph is \ch{connected} a.s.~if and only if
\ch{the corresponding group has polynomial growth of degree at most~$4$}~\cite{HebSaCo93,LP:book}.


Besides connectivity, several other basic features of the WUSF are also understood rather firmly\tom{. This understanding mostly stems from Wilson's algorithm rooted at infinity, which allows the WUSF to be sampled by joining together loop-erased random walks \cite{Wilson96,LP:book}}. For example, \tom{other than connectivity, the simplest property of a forest  is the number of {\em ends} its components have.} Here, an infinite graph $G$ is said to be $k$-ended if, over all finite sets of vertices $W$, the subgraph induced by $V\setminus W$ has a maximum of $k$ infinite connected components. In particular, an infinite tree is one-ended if and only if it does not contain a simple bi-infinite path.  Following earlier work by Pemantle \cite{Pem91}, BLPS \cite{BLPS} proved that the number of components of the WUSF of any graph is non-random, that the  WUSF of any unimodular transitive graph (e.g., any Cayley graph) is either connected or has infinitely many components a.s., and that in both cases every component of the WUSF is one-ended a.s.\ unless the underlying graph is itself two-ended.  Morris \cite{Morris03} later proved that every component of the WUSF is recurrent a.s.~on any graph, confirming a conjecture of BLPS  \cite[Conjecture 15.1]{BLPS}, and several other classes of graphs have also been shown to have one-ended WUSF components \cite{LMS08,AL07,H15}.

\ch{Much less is known about the FUSF.} No characterisation of its connectivity is known, nor  is it known whether the number of components of the FUSF is non-random on \ch{an any} graph.
%
%
In \cite{BLPS} it is proved that if the FUSF and WUSF differ on a unimodular transitive graph, then a.s.~the FUSF has a transient tree with infinitely many ends, in contrast to the WUSF. However, it remained an open problem \cite[Question 15.8]{BLPS} to prove that, under the same hypotheses, \emph{every} connected component of the FUSF is transient and infinitely ended a.s.
\ch{
In light of this, it is natural to ask the following more general question:

\begin{question*} Let $G$ be a unimodular transitive graph. Can the components of the free uniform spanning forest of $G$ be very different from each other? \end{question*}

Questions of this form were first studied by Lyons and Schramm \cite{LS99} in the context \emph{insertion-tolerant}  automorphism-invariant random subgraphs. Their remarkable theorem asserts that in any such random subgraph (e.g. Bernoulli bond percolation or the Fortuin-Kasteleyn random cluster model) on a unimodular transitive graph, one cannot distinguish between the infinite connected components using automorphism-invariant graph properties. For example, all such components must have the same volume growth, spectral dimension, value of $p_c$ and so forth (see Section \ref{comppropexamples} for further examples). They also exhibited applications of indistinguishability to statements not of this form, including uniqueness monotonicity and connectivity decay.
Here, a random subgraph $\omega$ of a graph $G$ is insertion-tolerance if for every edge $e$ of $G$, the law of the subgraph $\omega \cup \{e\}$ formed by inserting $e$ into $\omega$ is absolutely continuous with respect to the law of $\omega$.
The  uniform spanning forests are clearly not insertion-tolerant, since the addition of an edge may close a cycle.}

BLPS conjectured \cite[Conjecture 15.9]{BLPS} that the components of both the WUSF and FUSF also \ch{exhibit this form of indistinguishability}.
%
%
In this paper we confirm this conjecture.
\begin{thm}[Indistinguishability of USF components]\label{mainthmsimple} Let $G$ be a unimodular transitive graph, and let $\F$ be a sample of either the free uniform spanning forest or the wired uniform spanning forest of $G$. Then for each automorphism-invariant Borel-measurable set $\sA$ of subgraphs of $G$, either every connected component of $\F$ is in $\A$ or every connected component of $\F$ is not in $\A$ almost surely.
\end{thm}
As indicated by the above discussion, \cref{mainthmsimple} implies the following positive answer to \cite[Question 15.8]{BLPS} under the assumption of unimodularity.
\begin{thm}[Transient trees in the FUSF]\label{T:everyFUSFtreeistransientsimple} Let $G$ be a unimodular transitive graph and let $\F$ be a sample of $\FUSF_G$. If the measures $\FUSF_G$ and $\WUSF_G$ are distinct, then every component of $\F$ is transient and has infinitely many ends almost surely. \end{thm}
However, rather than deducing \cref{T:everyFUSFtreeistransientsimple} from \cref{mainthmsimple}, we instead prove \cref{T:everyFUSFtreeistransientsimple} directly and apply it in the proof of \cref{mainthmsimple}.

\medskip
\ch{We also apply \cref{mainthmsimple} to answer another of the most basic open problems about the FUSF \cite[Question 15.6]{BLPS} under the assumption of unimodularity.}
\begin{thm}[Number of trees in the FUSF]\label{1inftycomponentssimple} Let $G$ be a unimodular transitive graph  and let $\F$ be a sample of the free uniform spanning forest of $G$. Then $\F$ is either connected or has infinitely many components almost surely. \end{thm}

\ch{The derivation of \cref{1inftycomponentssimple} from \cref{mainthmsimple}} is inspired by the proof of \cite[Theorem 4.1]{LS99}, and also establishes the following result.

\begin{thm}[Connectivity decay in the FUSF]\label{T:Longrangedisordersimple} Let $G$ be a unimodular transitive graph and let $\F$ be a sample of the free uniform spanning forest of $G$. If $\F$ is disconnected a.s., then for every vertex $v$ of $G$,
\[\inf\left\{\FUSF_G(u\in T_\F(v)) : u \in V(G)\right\}=0 \, , \]
\ch{where $u\in T_\F(v)$ is the event that $u$ belongs to the component of $v$ in $\F$.}
\end{thm}

We prove all of our results in the much more general setting of \emph{unimodular random rooted networks}, which includes all Cayley graphs as well as a wide range of popular infinite random graphs and networks \cite{AL07}. \ct{For example, our results hold when the underlying graph is an infinite supercritical percolation cluster in a Cayley graph,  a hyperbolic unimodular random triangulation \cite{PSHIT, BPP14} \ct{(for which the FUSF and WUSF are shown to be distinct in the upcoming work \cite{unimodular2})}, a \ct{supercritical} Galton-Watson tree, or even a component of the \ct{FUSF of another unimodular random rooted network.}}
See \cref{S:compprop} for the strongest \ch{and most general} statements.

\medskip

\ca{
\noindent{\bf Organization.} In \cref{S:about} we describe our approach and the novel ingredients of our proof. The necessary background, including definitions of USFs and \URNs are presented in \cref{S:background}. In \cref{S:compprop} we define the graph properties we will work with, state the most general and strongest versions of our theorems (most importantly, \cref{mainthm}), and provide several illustrative examples.
In \cref{S:FUSFindist} we develop the update-tolerance property of the FUSF, and prove\ct{, in the setting of \cref{mainthm}, that if the FUSF and WUSF are distinct} then every component of the FUSF is transient and infinitely-ended (\cref{T:everyFUSFtreeistransient}), and then prove indistinguishability of the components in this case. In \cref{S:1infty}, still in the case where the FUSF and WUSF are distinct, we prove that the FUSF \ct{is either connected or has infinitely many connected components} (\cref{1inftycomponents}) and in the latter case we show that connectivity decay is exhibited (\cref{T:Longrangedisorder}). In \cref{S:wusfindist} we show that the WUSF components are indistinguishable, completing the proof of \cref{mainthm}.
}

\medskip

\noindent {\bf Remark.} After  this paper was posted on the arXiv, Adam Tim\'ar posted independent work \cite{Timar15} in which he proves \cref{mainthmsimple} for the FUSF only in the case that $\FUSF\neq\WUSF$,  and also proves Theorems \ref{T:everyFUSFtreeistransientsimple} and \ref{1inftycomponentssimple}; Indistinguishability of components in the WUSF is not treated. In the present paper, we prove \cite[Conjecture 15.1]{BLPS} in its entirety for both the FUSF and WUSF.

\subsection{About the proof}\label{S:about}

In \cite{LS99}, Lyons and Schramm argue that the coexistence of clusters of different types \ca{in an invariant edge percolation} implies the existence of infinitely many \emph{pivotal edges}, that is, closed edges that change the type of an infinite cluster if they are inserted. \ca{When the percolation is insertion-tolerant,} this heuristically contradicts the Borel-measurability of the property, as the existence of pivotal edges far away from the origin should imply that we cannot approximate the event that the cluster has the property by a cylinder event. This argument was made precise in~\cite{LS99}. \ct{Unimodularity of the underlying graph was used heavily  -- indeed, indistinguishability can fail without it \cite[Remark 3.16]{LS99}.}

A crucial ingredient of our proof is an \emph{update-tolerance} property of USFs. This property was introduced for the WUSF by the first author \cite{H15} and is developed for the FUSF in \cref{S:cyclebreakingintheFUSF}. This property allows us to make a local modification to a sample of the FUSF or WUSF in such a way that the law of the resulting modified forest is absolutely continuous with respect to the law of the forest that we started with. \ca{In this local modification, we add an edge of our choice to the USF and, in exchange, are required to remove an edge emanating from the same vertex.
 The edge that we are required to remove is random and depends upon both the edge we wish to insert and on
 the entire sample of the USF.}

Update-tolerance replaces insertion-tolerance and allows us to perform a variant of the key argument in \cite{LS99}. However, several obstacles arise as we are required to erase an edge at the same time as inserting one. In particular, we cannot simply open a closed edge connecting two clusters of different types in order to form a single cluster.
These obstacles are particularly severe for the WUSF (and the FUSF in the case that the two coincide), where it is no longer the case that the coexistance of components of different types implies the existence of pivotal edges.  To proceed, we separate the component properties into two types, tail and non-tail, according to whether the property is sensitive to finite modifications of the component. Indistinguishability is then proven by a different argument in each case: non-tail properties are handled by a variant of the Lyons-Schramm method, while tail properties are handled by a completely separate argument utilising Wilson's algorithm \cite{Wilson96} and the spatial Markov property. The proof that components of the WUSF cannot be distinguished by tail properties also applies to transitive graphs without the assumption of unimodularity.

\subsection{Background and Definitions}\label{S:background}


\subsubsection{Notation}

A {\bf tree} is a connected graph with no cycles. A {\bf spanning tree} of a graph $G=(V,E)$ is a connected subgraph of $G$ that contains every vertex and no cycles. A {\bf forest} is a graph with no cycles, and a {\bf spanning forest} of a graph $G=(V,E)$ is a subgraph of $G$ that contains every vertex and no cycles. Given a forest $\F$ and a vertex $v$ we write $T_\F(v)$ for the connected component of $\F$ containing $v$. An {\bf essential} spanning forest is a spanning forest such that every component is infinite. A {\bf branch} of an infinite tree $T$ is an infinite component of $T \setminus v$ for some vertex $v$. The \textbf{core} of an infinite tree $T$, denoted $\core(T)$, is the set of vertices of $T$ such that $T\setminus v$ at least two infinite connected components.

Recall that an infinite graph $G$ is said to be \textbf{$k$-ended} if removing a finite set of vertices $W$ from $G$ results in a maximum of $k$ distinct infinite connected components. In particular, an infinite tree is one-ended if and only if it does not contain any simple bi-infinite paths. We say that a forest $\F$ is one-ended if all of its components are one-ended. The {\bf past} of a vertex $v$ in a one-ended forest, denoted $\past_\F(v)$, is the union of \asaf{$v$ and the} finite components of $\F \setminus v$. The \textbf{future} of the vertex $v$ is the set of $u$ such that $v\in \past_\F(u)$.

We write $B_G(v,r)$ for the graph-distance ball of radius $r$ about a vertex $v$ in a graph $G$.
\subsubsection{Uniform Spanning \ca{Trees and} Forests}\label{S:USFbackground}

\ca{We now briefly provide the necessary definitions, notation and background concerning USFs. We refer the reader to \cite[\S4 and \S10]{LP:book} for a comprehensive review of this theory.}
\ca{Given a graph $G=(V,E)$ we will refer to an edge $e\in E$ both as an oriented and unoriented edge and it will always be clear which one from the context. Most frequently we will deal with oriented edges and in this case we orient them from their tail $e^-$ to their head $e^+$.}
%

A \textbf{network} $(G,c)$ is a locally finite, connected multi-graph $G=(V,E)$ together with a function $c:E\to(0,\infty)$ assigning a positive \textbf{conductance} to each edge of $G$. Graphs are considered to be networks by setting $c \equiv 1$. The distinction between graphs and networks does not play much of a role for us, and we will mostly suppress the notation of conductances, writing $G$ to mean either a graph or a network.
Write $c(u)$ for the sum of the conductances of the edges $e^-=u$ emanating from $u$ and $c(u,v)$ for the \ca{conductance of the sum of the conductances of the (possibly many) edges with endpoints $u$ and $v$.} The \textbf{random walk} $\langle X_n \rangle_{n\geq 0}$ on a network $G$ is the Markov chain on $V$ with transition probabilities $p(u,v) = c(u,v)/c(u)$.

 The \textbf{uniform spanning tree} measure $\UST_G$ of a finite \ca{connected} graph $G$ is the uniform measure on spanning trees of $G$ (considered for measure-theoretic purposes as functions $E\to\{0,1\}$).
   When $G$ is a network, $\UST_G$ is the probability measure on spanning trees of $G$ such that the probability of a tree is proportional to the product of the conductances of its edges.

Let $G$ be an infinite network. An \textbf{exhaustion }$\langle V_n\rangle_{n\geq0}$ of $G$ is an increasing sequence of finite sets of vertices $V_n\subset V$ such that $\bigcup_{n\geq0} V_n =V$. Given such an exhaustion, we define $G_n$ to be the subgraph of $G$ induced by $V_n$ together with the conductances inherited from $G$, and define  $G_n^*$ to be the network obtained from $G$ by identifying (or ``wiring'') $V\setminus V_n$ into a single vertex and deleting all the self-loops that are created.
The weak limits of the measures $\UST_{G_n}$ and $\UST_{G_n^*}$ exist for any network and do not depend on the choice of exhaustion \cite{Pem91,Hagg95}. The limit of the $\UST_{G_n}$ is called the \textbf{free uniform spanning forest} measure $\FUSF_G$ while the limit of the $\UST_{G_n^*}$ is called the \textbf{wired uniform spanning forest} measure $\WUSF_G$. Both limits are clearly concentrated on the set of essential spanning forests of $G$.

\ct{The measures $\FUSF_G$ and $\WUSF_G$ coincide if and only if $G$ does not support any non-constant harmonic functions of finite Dirichlet energy \cite{BLPS}, and in particular the two measures coincide when $G=\Z^d$. The two measures also coincide on every amenable transitive graph \cite[Corollary 10.9]{BLPS}, and an analogous statement holds for unimodular random rooted networks once an appropriate notion of amenability is adopted \cite[\S8]{AL07}.
When $G$ is a Cayley graph, the two measures $\FUSF_G$ and $\WUSF_G$ coincide if and only if the \emph{first $\ell^2$-Betti number} of the corresponding group is zero \cite{Lyons09}. \ca{By taking various free or direct products of groups and estimating their Betti numbers, this characterization allows to construct an abundance of Cayley graphs in which the two measures either coincide or differ} \cite[\S10.2]{LP:book}.}

\medskip

\ca{A very useful property of the \ct{UST and the} USFs is \ct{the} {\em spatial Markov property}.} Let $G$ be a network and let $H$ and $F$ be finite subsets of $G$. We write $\hat G = (G-H)/F$ for the network formed from $G$ by deleting each edge $h \in H$ and contracting (i.e., identifying the two endpoints of) each edge $f \in F$.
If $G$ is finite and $T$ is a sample of $\UST_G$, then the law of $T$ conditioned on the event $\{F \subseteq T, H\cap T=\emptyset\}$ (assuming this event has positive probability) is equal to the law of the union of $F$ with an independent copy of $\UST_{\hat G}$, considered as a subgraph of $G$ \cite[\S4]{LP:book}.
Now suppose that $G$ is an infinite network with exhaustion $\langle V_n \rangle_{n\geq0}$ and let $\F$ be a sample of either $\FUSF_G$ or $\WUSF_G$. Applying the Markov property to the finite networks $G_n$ and $G_n^*$ and taking the limit as $n\to\infty$, we see similarly that the conditional distribution of $\F$ conditioned on the event $\{F \subseteq \F, H\cap \F=\emptyset\}$ is equal to the law of the union of $F$ with an independent copy of $\FUSF_{\hat G}$ or $\WUSF_{\hat G}$ as appropriate. It is important here that $H$ and $F$ are finite.

\medskip

\ca{Lastly, throughout \cref{S:wusfindist} we will use a recent result of the first author regarding ends of the WUSF's components.}
%
Components of the WUSF are known to be one-ended a.s.~in several large classes of graphs and networks. The following is proven by the first author in \cite{H15}, and follows earlier works \cite{AL07, BLPS, LMS08, Pem91}.
\begin{thm} [\cite{H15}] \label{oneend} Let $(G,\rho)$ be transient unimodular random rooted network with $\E [c(\rho)] < \infty$. Then every component of the wired uniform spanning forest of $G$ is one-ended almost surely.
\end{thm}

\subsubsection{Unimodular random \ca{networks}}\label{S:unimodular}

We present here the necessary definition of unimodular random networks and refer the reader to the comprehensive monograph of Aldous and Lyons \cite{AL07} for more details and many examples. A \textbf{rooted graph} $(G,\rho)$ is a locally finite, connected graph $G$ together with a distinguished vertex $\rho$, the root. An isomorphism of graphs is an isomorphism of rooted graphs if it preserves the root. The ball of radius $r$ around a vertex $v$ of $G$, denoted $B_G(v,r)$ is the graph induced on the set of vertices which are at graph distance at most $r$ from $v$.
The \textbf{local topology} on the set of isomorphism classes of rooted graphs is defined so that two (isomorphism classes of) rooted graphs $(G,\rho)$ and $(G',\rho')$ are close to each other if and only if the rooted balls $(B_G(\rho,r),\rho)$ and $(B_{G'}(\rho',r),\rho')$ are isomorphic to each other for large $r$. We denote the space of isomorphism classes of rooted graphs endowed with the local topology by $\cG_\bullet$.
We define an \textbf{edge-marked graph} to be a locally finite connected graph together with a function $m:E(G)\to\mathbb{X}$ for some separable metric space $\mathbb{X}$, the \textbf{mark space} (in this paper, $\mathbb{X}$ will be a product of intervals and some copies of $\{0,1\}$). For example, if $G=(G,c)$ is a network and $\F$ is a sample of $\FUSF_G$, then $(G,c,\F)$ is a graph with marks in $(0,\infty)\times\{0,1\}$.
The local topology on rooted marked graphs is defined so that two marked rooted graphs are close if for large $r$ there is an isomorphism (of rooted graphs) $\phi: (B_G(\rho,r), m, \rho) \to (B_{G'}(\rho',r),\rho')$ such that $d_{\mathbb{X}}(m'(\phi(e)),m(e))$ is small for every edge $e$ in $B_G(\rho,r)$.
We denote the space of edge-marked graphs with marks in $\mathbb{X}$ by $\cG^\mathbb{X}_\bullet$.

Similarly, we define a \textbf{doubly-rooted graph} $(G,u,v)$ to be a graph together with an ordered pair of distinguished vertices.
The space $\cG_{\bullet\bullet}$ of doubly-rooted graphs is defined similarly to $\cG_\bullet$.
A random rooted graph $(G,\rho)$ is {\bf unimodular} if it obeys the {\bf mass-transport principle}. That is, for every non-negative Borel function $f:\Gg \to [0,\infty]$ -- which we call a \textbf{mass transport} -- we have that
\[ \E \sum_{v \in V} f(G,\rho,v) = \E \sum_{u\in V} f(G, u, \rho).\]
In other words, $(G,\rho)$ is unimodular if for every mass transport $f$, the expected mass received by the root equals the expected mass sent by the root.
Every Cayley graph (rooted at any vertex) is a  unimodular random rooted graph (whose law is concentrated on a singleton), as is every unimodular transitive graph \cite[\S8]{LP:book}. For many examples of a more genuinely random nature, see \cite{AL07}. Unimodular random rooted networks and other edge-marked graphs are defined similarly.

When $(G,\rho)$ is a unimodular random rooted network and $\F$ is a sample of either $\FUSF_G$ or $\WUSF_G$, $(G,\rho,\F)$ is also unimodular: Since the definitions of $\FUSF_G$ and $\WUSF_G$ do not depend on the choice of exhaustion,
for each mass transport $f:\cG_{\bullet\bullet}^{(0,\infty)\times \{0,1\}}\to[0,\infty]$, the expectations
\begin{align*}
f^F(G,u,v) = \FUSF_G\left[f(G,u,v,\F) \right] \quad\text{ and }\quad f^W(G,u,v)= \WUSF_G\left[f(G,u,v,\F) \right]
 \end{align*}
are also mass transports.  This allows us to deduce the mass-transport principle for $(G,\rho,\F)$ from that of $(G,\rho)$.

\subsubsection{Reversibility and stationarity}\label{S:reversibility}

Let $(G,\rho)$ be a random rooted network and let $\langle X_n \rangle_{n\geq0}$ be a random walk on $G$ started at $\rho$. The random rooted graph $(G,\rho)$ is said to be \textbf{stationary} if
\[ (G,\rho) \eqd (G,X_1) \]
and \textbf{reversible} if
\[ (G,\rho,X_1) \eqd (G,X_1,\rho). \]
While every reversible random rooted graph is trivially stationary, the converse need not hold in general. Indeed, every transitive graph (rooted arbitrarily) is stationary, while it is reversible if and only if it is unimodular. For example, the \emph{grandfather graph} \cite{LP:book} is transitive but not reversible.

The following correspondence between unimodular and reversible random rooted networks is implicit in \cite[\S4]{AL07} and is proven explicitly in \cite{BC2011}.
\begin{prop}\label{P:unimodularreversible} If $(G,\rho)$ is a unimodular random rooted network with $\E[c(\rho)] <\infty$, then biasing the law of $(G,\rho)$ by $c(\rho)$ (that is, reweighting the law of $(G,\rho)$ by the Radon-Nikodym derivative $c(\rho)/\E[c(\rho)]$) yields the law of a reversible random rooted network. Conversely, if $(G,\rho)$ is a reversible random rooted network with $\E[c(\rho)^{-1}]<\infty$ then biasing the law of $(G,\rho)$ by $c(\rho)^{-1}$ yields the law of a unimodular random rooted network.
\begin{equation*}\label{eq:reversibleunimodularcorrespondance} \left\{\begin{array}{l} (G,\rho) \text{ unimodular}\\ \text{with }\E[c(\rho)]<\infty\end{array}\right\} \begin{array}{l}\xrightarrow{\hspace{.45em}\text{bias by }c(\rho)\hspace{.45em}} \\ \xleftarrow[\text{bias by }c(\rho)^{-1}]{} \end{array}  \left\{\begin{array}{l}\hspace{0.75em} (G,\rho) \text{ reversible}\\ \text{with }\E[c(\rho)^{-1}]<\infty\end{array}\right\}.\end{equation*}
\end{prop}
For example, a finite rooted network is unimodular if and only if, \ca{conditioned on $G$}, its root is uniformly distributed on the network, and is reversible if and only if, \ca{conditioned on $G$}, the root is distributed according to the stationary distribution of the random walk on the network.

Thus, to prove an almost sure statement about unimodular random rooted networks with $\E[c(\rho)]<\infty$ we can bias by the conductance at the root and work in the reversible setting, and vice versa.

A useful equivalent characterisation of reversibility is as follows. Let $\mathcal{G}_\leftrightarrow$
denote the space of isomorphism classes of graphs equipped with a
bi-infinite path $(G,\langle x_n\rangle_{n \in \Z})$, which is endowed with a natural variant of the local topology.  Let $(G,\rho)$ be a random rooted graph and let $\langle X_n \rangle_{n\geq 0}$ and $\langle X_{-n}\rangle_{n\geq 0}$ be two
independent simple random walks started from $X_0=\rho$, so that $(G,\langle X_n \rangle_{n \in \Z})$ is a random variable taking values in $\mathcal{G}_\leftrightarrow$. Then $(G,\rho)$ is reversible if and only if
\begin{equation}\label{Eq:reversibility}(G,\langle X_n \rangle_{n\in \Z}) \eqd (G,\langle X_{n+k} \rangle_{n \in \Z}) \quad \forall\, k \in \Z. \end{equation}
Indeed, $(\rho,X_{-1},\ldots)$ is a simple random walk
started from $\rho$ independent of $X_1$ and, conditional on $(G,X_1)$,
reversibility implies that $\rho$ is uniformly distributed among the
neighbours of $X_1$, so that $(X_1,\rho,X_{-1},X_{-2},\ldots)$ has the law of a
simple random walk from $X_1$ and \eqref{Eq:reversibility} follows. Conversely, \eqref{Eq:reversibility} implies that $(G,\rho)$ is reversible by taking $k=1$ and restricting to the 0th and 1st coordinates of the walk.

A useful variant of \cref{P:unimodularreversible} is the following. Suppose that $(G,\rho)$ is a unimodular random rooted network with $\E[c(\rho)]<\infty$, $\F$ is a sample of either $\FUSF_G$ or $\WUSF_G$, and let $c_\F(v)$ denote the sum of the conductances of the edges of $G$ emanating from $v$ that are included in $\F$. Then, if we sample $(G,\rho,\F)$ biased by $c_\F(\rho)$ and let $\langle X_n \rangle_{n\geq0}$ and $\langle X_{-n}\rangle_{n\geq0}$ be independent random walks on $\F$ starting at $\rho$, then, by \cite[Theorem 4.1]{AL07},
 \begin{equation}\label{Eq:reversibleforest0}(G,\langle X_{n} \rangle_{n\in \Z},\F) \eqd (G,\langle X_{n+k} \rangle_{n\in \Z},\F) \quad \forall\, k \in \Z. \end{equation}
\subsubsection{Ergodicity}\label{S:ergodicity}

We say that a unimodular random rooted network with $\E[c(\rho)]<\infty$ is \textbf{ergodic} if any (and hence all) of the below hold.
\begin{thm}[Characterisation of ergodicity {\cite[\S 4]{AL07}}]
  Let $(G,\rho)$ be a unimodular random rooted network with $\E[c(\rho)]<\infty$. The following are
  equivalent.
  \begin{enumerate}\itemsep0em
  \item When the law of $(G,\rho)$ is biased by $c(\rho)$ to give an equivalent reversible random rooted network, the stationary sequence $\langle (G,X_n) \rangle_{n\geq 0}$ is ergodic.
  \item Every event $A \subset \mathcal{G}_\bullet^{(0,\infty)}$ invariant to
    changing the root has probability in $\{0,1\}$.
  \item The law of $(G,\rho)$ is an extreme point of the weakly closed
    convex set of laws of unimodular random rooted networks.
\end{enumerate}
\end{thm}
A similar statement holds for edge-marked networks. Tail triviality of the USFs \cite[Theorem 8.3]{BLPS} implies that if $(G,\rho)$ is an ergodic unimodular random rooted network and $\F$ is a sample of either $\FUSF_G$ or $\WUSF_G$, then $(G,\rho,\F)$ is also ergodic.

The extremal characterisation $(3)$ implies (by Choquet theory) that every unimodular random rooted network with $\E[c(\rho)]<\infty$ can be written as a mixture of ergodic unimodular random rooted networks. Thus, to prove a.s.~statements about general unimodular random rooted networks it suffices for us to consider ergodic unimodular random rooted networks.

\subsection{Component properties and indistinguishability on \URNs}\label{S:compprop}
General unimodular random rooted graphs and networks have few automorphisms, so that it is not appropriate at this level of generality to phrase indistinguishability in terms of automorphism-invariant properties.
Instead, we consider properties that are invariant under rerooting within a component as follows.
Consider the space $\cG^{\{0,1\}}_\bullet$ of rooted graphs with edges marked by $\omega(e)\in\{0,1\}$, which we think of as a rooted graph together with a distinguished subgraph spanned by the edges $\omega=\{e:\omega(e)=1\}$.
Given such a $(G,v,\omega)$ we define $K_\omega(v)$ to be the connected component of $v$ in $\omega$.

\begin{defn} A Borel-measurable set $\A \subset \Gb$ is called a \textbf{component property} if and only if it is invariant to rerooting within the component of the root, i.e.,
$$ (G, v, \omega) \in \A \Longrightarrow (G, u, \omega) \in \A \,\,\, \forall u \in K_\omega(v) \, .$$
\end{defn}
Again, this may be formulated for networks with the obvious modifications. This definition is equivalent to the one given in \cite[Definition 6.14]{AL07}.
We say that a connected component $K$ of $\omega$ has property $\sA$ (and abuse notation by writing $K \in \sA$) if $(G,u,\omega)\in \sA$ for some (and hence every) vertex $u\in K$. We are now ready to state our main theorem in its full generality and strength.

\begin{thm}[Indistinguishability of USF components] \label{mainthm} Let $(G,\rho)$ be a unimodular random network with $\E[c(\rho)]<\infty$, and let $\F$ be a sample of either $\FUSF_G$ or $\WUSF_G$. Then for every component property $\A$, either every connected component of $\F$ has property $\A$ or none of the connected components of $\F$ have property $\sA$ almost surely.
\end{thm}


And we may now restate Theorems \ref{1inftycomponentssimple}, \ref{T:Longrangedisordersimple} and \ref{T:everyFUSFtreeistransientsimple} in their full generality.

\begin{thm}\label{1inftycomponents} Let $(G,\rho)$ be a unimodular random rooted network with $\E[c(\rho)]<\infty$ and let $\F$ be a sample of $\FUSF_G$. Then $\F$ is either connected or has infinitely many components almost surely. \end{thm}

\begin{thm}\label{T:Longrangedisorder} Let $(G,\rho)$ be a unimodular random rooted network with $\E[c(\rho)]<\infty$ and let $\F$ be a sample of $\FUSF_G$. If $\F$ is disconnected a.s., then a.s.~for every vertex $v$ of $G$,
\[\inf\{\FUSF_G(u\in T_\F(v)) : u \in V(G)\}=0. \]
\end{thm}

\begin{thm}\label{T:everyFUSFtreeistransient} Let $(G,\rho)$ be a unimodular random rooted network and let $\F$ be a sample of the $\FUSF_G$. On the event that the measures $\FUSF_G$ and $\WUSF_G$ are distinct, every component of $\F$ is transient and has infinitely many ends almost surely. \ca{This holds both when the edges of $\F$ are given \ct{the} conductances inherited from $G$ and when they are given unit conductances.} \end{thm}

\noindent It follows that, under the assumptions of \cref{T:everyFUSFtreeistransient}, every component of the FUSF of $G$ has positive speed and critical percolation probability $p_c<1$  \cite{AL07}.

\medskip

We remark that, by  \cite[Proposition 5]{GabLyons07},  \cref{mainthm} is equivalent to the following ergodicity statement.

\begin{corollary}\label{c:ergodicity} Let $(G,\rho)$ be an ergodic unimodular random rooted network with $\E[c(\rho)]<\infty$ and let $\F$ be a sample of either $\FUSF_G$ or $\WUSF_G$.
Then $(T_\F(\rho),\rho)$ is an ergodic unimodular random rooted network.
 Moreover, if we bias the distribution of  $(G,\rho,\F)$ by $c_\F(\rho)$ and let $\langle X_n \rangle_{n\geq0}$ and $\langle X_{-n}\rangle_{n\geq0}$ be independent random walks on $\F$ started at $\rho$, then the stationary sequence
$\left\langle (G,\langle X_{n+k} \rangle_{n\in \Z},\F)\right\rangle_{k\in \Z}$ is ergodic.
\end{corollary}


\subsubsection{Examples of component properties}\label{comppropexamples}

\begin{example}[Automorphism-invariant properties] Let $G_0$ be a transitive graph, and let $\sA$ be an automorphism-invariant set of subgraphs of $G_0$, \ca{that is, $\gamma \sA = \sA$ for any automorphism $\gamma$} of $G_0$.
Fix an arbitrary vertex $v_0$ of $G_0$ and let
\[ \sA' = \left\{(G,v,\omega) : \begin{minipage}{0.495\linewidth}
$\exists \text{ an isomorphism } \phi:(G,v) \to (G_0,v_0)\\
 \text{such that } \phi(K_\omega(v))\in\sA$
 \end{minipage}\right\}\]
Then $\sA'$ is a component property such that $(G_0,v_0,\omega)\in\sA'$ if and only if $K_\omega(v_0)\in\sA$. \ca{Thus, \cref{mainthmsimple} follows from \cref{mainthmsimple} and similarly Theorems \ref{1inftycomponentssimple}, \ref{T:Longrangedisordersimple} and \ref{T:everyFUSFtreeistransientsimple} follow from Theorems \ref{1inftycomponents}, \ref{T:Longrangedisorder} and \ref{T:everyFUSFtreeistransient}, respectively.}
\end{example}


\begin{example}[Intrinsic properties] A graph $H$ is said to have \textbf{volume-growth dimension $d$} if
\[ |B_H(v,r)| = r^{d+o(1)} \]
\ca{for any (and hence all) vertices $v$ of $H$.} Let $p_n(\cdot,\cdot)$ denote the $n$-step transition probabilities of simple random walk on $H$. We say that $H$ has \textbf{spectral dimension} $d$ if
\[ p_n(v,v) = n^{-d/2+o(1)}\]
for any (and hence all) vertices $v$ of $H$. \ca{Lastly, recall that the \textbf{critical percolation probability} $p_c(H)$ of an infinite connected graph $H$ is the supremum over $p\in [0,1]$ such that independent percolation with edge probability $p$ a.s.~does not exhibit an infinite cluster.} Then, under the hypotheses of \cref{c:ergodicity}, the volume-growth, spectral dimension of $T_\F(\rho)$ (if they exist) \ca{and value of $p_c$} are non-random, and consequently are a.s.~the same for every tree in $\F$.
\end{example}

Component properties can be `extrinsic' and depend upon how the component sits inside of the base graph $G$.
\begin{example}[Extrinsic properties] A \ca{subgraph} $H$ of \ca{$G$} is said to have \textbf{discrete Hausdorff dimension} $\alpha$ if
\[ |B_G(v,n) \cap H| = n^{\alpha+o(1)} , \]
\ca{for any (and hence all) vertices $v$ of $G$}. The event that a component has a particular discrete Hausdorff dimension is a component property, and consequently \cref{mainthm} implies that all components of the USF in a \ca{\URN} have the same discrete Hausdorff dimension (if this dimension exists). In fact, the discrete Hausdorff dimension of every component of the USF in $\Z^d$ was proven to be $4$ for all $d\geq4$ by Benjamini, Kesten, Peres and Schramm~\cite{BeKePeSc04}.
\end{example}

Even for unimodular transitive graphs, the conclusion of \cref{mainthm} is strictly stronger than that of \cref{mainthmsimple}. This is because a component property can also depend on the whole configuration, as the following example demonstrates.
\begin{example}
Define $N(v,\omega,r)$ to be the number of distinct components of $\omega$ that are adjacent to the ball $B_{\omega}(v,r)$ of radius $r$ about $v$ in the intrinsic distance on $K_\omega(\rho)$. Asymptotic statements about the growth of $N(v,\omega,r)$, can be used to define component properties that depend on the entire configuration $\omega$, e.g.
\[ \sA = \{ (G,v,\omega) :  N(v,\omega,r) = r^{\beta+o(1)} \}. \]
\end{example}

\ct{All of the examples above are what we call {\em tail} properties. That is, these properties can be verified by looking at all but finitely many edges of both the component and of the configuration (see the next section for the precise definition). Let us give now an interesting example of a non-tail property.}

%

\begin{example}[Non-tail property] The component property
\[\sA(n)= \left\{(G,v,\omega) : \text{\begin{minipage}{0.65\textwidth}
for each connected component $K$ of $\omega$, there exists a path $\langle e_i \rangle$ in $G$ connecting $K_\omega(v)$ to $K$ such that at most $n$ of the $e_i$\ca{'s} are not contained in $\omega$.
 \end{minipage}}\right\} \]
is not a tail property. Benjamini, Kesten, Peres and Schramm \cite{BeKePeSc04} proved the remarkable result that the property $\sA(n) \setminus \sA(n-1)$ holds a.s.~for every component of the USF of $\Z^d$ if and only if $4(n-1)< d\leq 4n$.
\end{example}

\subsubsection{Tail properties}
\begin{defn}
We say that a component property $\A$ is a \textbf{tail component property} if
\[ (G,v,\omega)\in \A \Longrightarrow (G,v,\omega')\in \sA \,\,\,
\begin{array}{l}
\forall \omega'\subseteq E(G) \text{ such that } \omega \symdif \omega' \text{ and }\\
 K_\omega(v)\symdif K_{\omega'}(v) \text{ are both finite.}
\end{array}
\]
\end{defn}


Indistinguishability of USF components by properties that are not tail can fail without the assumption of unimodularity -- see~\cite[Remark~3.16]{LS99} and \cite[Example~3.1]{BLPS99}. However, our next theorem shows that unimodularity is not necessary for indistinguishability of WUSF components by tail properties when the WUSF components are a.s.~one-ended. In \cite{LMS08} it is shown that the last condition holds in every transient transitive graph. The following theorem, which is used in the proof of \cref{mainthm}, implies that WUSF components are indistinguishable by tail properties in any transient transitive graph (not necessarily unimodular).

\begin{thm}\label{tailproperty} Let $(G,\rho)$ be a stationary random network and let $\F$ be a sample of $\WUSF_G$. Suppose that every component of $\F$ is one-ended almost surely. Then for every tail component property $\A$, either every connected component of $\F$ has property $\A$ or none of the connected components of $\F$ have property $\sA$ almost surely.
\end{thm}

\ca{\subsubsection{Sharpness}}

\ca{
We present a construction showing that the condition $\E[c(\rho)]<\infty$ in \cref{mainthm} is indeed necessary. For integers $n$ and $k>2$, denote by $T_n(k)$ the finite network on a binary tree of height $n$ such that edges at distance $h$ from the leaves have conductance $k^h$. Choose a uniform random root in $T_n(k)$ and take $n$ to $\infty$ while keeping $k$ fixed. The limit of this process can be seen to be the transient \URN $T(k)$ in which the underlying graph is the canopy tree \ct{\cite{AizWar06}} and edges of distance $h$ from the leaves have conductance $k^h$.

Consider the finite network $G_n$ obtained by gluing a copy of $T_n(3)$ and a copy of $T_n(4)$ at their leaves \ct{in such a way that the resulting network $G_n$ is planar}, and let $\rho_n$ a uniformly chosen root vertex of $G_n$. Then the randomly rooted graphs $(G_n, \rho_n)$ converge to a \URN which is formed by gluing a copy of $T(3)$ and a copy of $T(4)$ at their leaves. It can easily be seen via Wilson's algorithm (see \cref{S:wusfindist}) that the WUSF will contain precisely two one-ended components corresponding to the two infinite rays of $T(3)$ and $T(4)$. These two trees are clearly distinguishable from each other by measuring the frequency of edges with conductances $3$ or $4$ on their infinite ray. This example can be made into a graph rather than a network simply by replacing an edge with conductance $k^h$ by $k^h$ parallel edges.

 }
 \ct{It is also possible to construct a unimodular random rooted network on which there are infinitely many WUSF components almost surely and every cluster is distinguishable from every other cluster. Let $G$ be a 3-regular tree, and let $\F_1$ be a sample of $\WUSF_T$. For every component $T$ of $\F_1$, let $U(T)$ be i.i.d. uniform $[0,1]$. For each edge $e$ of $G$ that is contained in $\F_1$, let $T(e)$ be the component of $\F_1$ containing $e$. Define conductances on $G$ by, for each edge $e$ of $G$, setting $c(e)=1$ if $e\notin\F_1$ and otherwise setting
 \[ c(e) = \exp\left((1+U(T))\times|\text{The finite component of }T(e)\setminus e|\right) \, .\]
 These strong drifts ensure that a random walk on the network $(G,c)$ will eventually remain in a single component of $\F_1$.  Running Wilson's algorithm on the network $(G,c)$ to sample a copy $\F_2$ of $\WUSF_{(G,c)}$, we see that the components of $\F_2$ correspond to the components of $\F_1$. We can distinguish these components from each other by observing the rate of growth of the conductances along a ray in each component.}
\medskip

\section{Indistinguishability of FUSF components}\label{S:FUSFindist}

\ca{Our goal in this section is to prove Theorem \ref{mainthm} for the FUSF on a \URN $(G,\rho)$ when the measures $\FUSF_G$ and $\WUSF_G$ are distinct.}

\subsection{Cycle breaking in the FUSF}\label{S:cyclebreakingintheFUSF}

Let $G$ be a finite network. For each spanning tree $t$ of $G$ and oriented edge $e$ of $G$ that is not a self-loop, we define the \textbf{direction} $D(e)=D(t,e)$ to be the first edge in the unique simple path from $e^-$ to $e^+$ in $t$.

\begin{lem} Let $G$ be an infinite network with exhaustion $\langle V_n \rangle_{n\geq1}$ and let $T_n$ be a sample of $\UST_{G_n}$ for each $n$. Then for every oriented edge $e$ of $G$, the random variables $(T_n,D(T_n,e))$ converge in distribution to some limit $(\F,D(e))$, where $D(e)$ is an edge adjacent to $e^-$ and the marginal distribution of $\F$ is given by $\FUSF_G$.
\end{lem}
\begin{proof}

 Since the distribution of $T_n$ converges  to $\FUSF_G$, it suffices to show that the conditional probabilities
\be \label{toconverge}
\P(D(T_n,e)=d \,|\,  f_1,\ldots,f_k \in T_n,\, h_1,\ldots h_l \notin T_n)
\ee
converge, where $d=(d^-,d^+)$ is any oriented edge with $d^-=e^-$ and $F=\{f_1,\ldots,f_k\}$ and $H=\{h_1,\ldots,h_l\}$ are any two finite collections of edges in $G$ for which the event $\mathscr{A}=\{f_1,\ldots,f_k \in \F,\, g_1,\ldots g_l \notin \F\}$ has non-zero probability. Fix such $d, F$ and $H$. If $F$ includes a path from $e^-$ to $e^+$ then $D(T_n,e)$ is determined and there is convergence in (\ref{toconverge}). Also, if $d \in H$ then (\ref{toconverge}) is zero. So let us assume now that neither is the case.

Let us first explain why convergence holds in (\ref{toconverge}) when $F=H=\emptyset$. In that case, by Kirchhoff's effective resistance formula (see \cite[Theorem 4.1]{BLPS}), the probability that the unique path between $e^-$ to $e^+$ in $T_n$ goes through $d$ equals the amount of current on the edge $d$ when a unit current flows from $e^-$ to $e^+$ in the network $G_n$. It is well known that this quantity converges as $n\to \infty$ to the current passing through $d$ in  the \emph{free unit current flow} from $e^-$ to $e^+$ in $G$ \cite[Proposition 9.1]{LP:book}.

When $F$ and $H$ are non-empty, we take $n$ to be sufficiently large such that the edges $e,d$ and all the edges of $F$ and $H$ are contained in $G_n$, and that the event $\mathscr{A}_n=\{F \subset T_n,\, H \cap T_n = \emptyset\}$ has non-zero probability (i.e, that $G_n \setminus H$ is connected and there are no cycles in $F$). We write $(G_n-H)/F$ for the network formed from $G_n$ by deleting each edge $h \in H$ and contracting each edge $f \in F$. By the Markov property (see \cref{S:USFbackground}), the laws of $T_n$ conditioned on $\A_n$ and of $\F$ conditioned on $\sA$ can be sampled from by taking the union of $F$ with a sample of the UST of $(G_n-H)/F$ or the FUSF of  $(G-H)/F$ respectively.
 Thus, the same argument as above works when $d \not \in F$ and shows that the limit of (\ref{toconverge}) is equal to the current passing through $d$ in the free unit current flow from $e^-$ to $e^+$ in $(G-H)/F$.

Finally suppose that $d\in F$. Let $V_d$ be the set of vertices connected to $e^-$ by a simple path in $F$ passing through $d$, and let $E_d$ be the set of oriented edges with tail in $V_d$. By our previous discussion, (\ref{toconverge}) equals the sum of the currents flowing through the edges of $E_d$ in the unit current flow from $e^-$ to $e^+$ in $(G_n - H)/F$. As before, \cite[Proposition 9.1]{LP:book} shows that this quantity converges to the corresponding sum of currents in the free unit current flow from $e^-$ to $e^+$ in $(G-H)/F$. \qedhere

\end{proof}


For each oriented edge $e$ of $G$ and $\FUSF_G$-a.e.~spanning forest $f$ of $G$, we define the \textbf{update} $U(f,e)$ as follows.
If $e$ is either a self-loop or already contained in $f$, set $U(f,e)=f$. Otherwise, sample $D(e)$ from its conditional distribution given $\F=f$, and set $U(f,e)=\F\cup\{e\}\setminus D(e)$. It seems likely that this conditional distribution is concentrated on a point.
\begin{question} Let $G$ be a network and let $e$ be an edge of $G$. Does $U(f,e)$ coincide $\FUSF_G$-a.e. with some measurable function of $f$?
\end{question}
 If any additional randomness is required to perform an update, it will always be taken to be independent of any other random variables considered.

\begin{lem}\label{Lem:FUSfupdatestationarity} Let $G$ be a network and $\F$ be a sample of $\FUSF_G$. Let $v$ be a vertex of $G$ and let $E$ be an element of the set $\{e: e^-=v \}$ chosen independently of $\F$ and with probability proportional to its conductance.
Then
$U(\F,E)$ and $\F$ have the same distribution.  \end{lem}

\begin{proof} Let $\langle V_n \rangle_{n\geq0}$ be an exhaustion of $G$ and let $T_n$ be a sample of the UST on $G_n$ for each $n$.
We may assume that $G_n$ contains $v$ and every edge adjacent to $v$ for all $n\geq 1$.
We define the update $U(t,e)$ of a spanning tree $t$ of $G_n$ at the oriented edge $e$ to be
\[ U(t,e) = t\cup\{e\}\setminus D(t,e).\]
 Since $U(T_n,E)$ converges to $U(\F,E)$ in distribution, and so it suffices to verify that $U(T_n,E) \eqd T_n$ for each $n\geq0$: this may be done by checking that $\UST_{G_n}$ satisfies the detailed balance equations for the Markov chain on the set of spanning trees of $G$ with transition probabilities
\[p(t_1,t_2)=\frac{1}{c(v)}c(\{{e} : {e^-}=v \text{ and } U({t_{{1}}},{e})={t_{{2}}}\}). \]
This simple calculation is carried out in \cite[Lemma 6]{H15}. \qedhere

\end{proof}

This has the following immediate consequence.

\begin{corollary}[Update Tolerance for the $\FUSF$]\label{updatetol} Let $G$ be a network. Fix an edge $e$, and let $\FUSF_G^e$ denote the joint distribution of a sample $\F$ of $\FUSF_G$ and of the update $U(\F,e)$. Then
\[ \FUSF_G(\F \in \A) \geq \frac{c(e)}{c(e^-)}\FUSF_G^e(U(\F,e) \in \A)\]
\end{corollary}

\begin{proof}
\cref{Lem:FUSfupdatestationarity} implies that
\begin{align*} \FUSF_G(\F \in \A) &= \frac{1}{c(e^-)}\sum_{\hat e^-=e^-}c(\hat e)\FUSF_G^{\hat e}(U(\F,\hat e) \in \A)\\
& \geq \frac{c(e)}{c(e^-)}\FUSF_G^e(U(\F,e) \in \A). \qedhere \end{align*}
\end{proof}


\subsection{All FUSF components are transient and infinitely-ended}

A \textbf{weighted tree} is a network whose underlying graph is a tree. Recall that a branch of an infinite tree $T$ is an infinite connected component of $T \setminus v$ for some vertex $v$.
\begin{lemma}\label{nobranches} Let $(T,\rho)$ be a unimodular random rooted weighted tree that is transient with positive probability. On the event that $T$ is transient, $T$ a.s.~does not have any recurrent branches.
\end{lemma}
\begin{proof}
For each vertex $u$ of $T$, let $V(u)$ be the set of vertices $v\neq u$ such that the component containing $u$ in $T \setminus v$ is recurrent. We first claim that if $T$ is a transient weighted tree and $u$ is a vertex of $T$, then $V(u)$ is either empty, or a finite simple path starting from \ca{a neighbo\ct{u}r of} $u$ in $T$, or an infinite transient ray (i.e. a ray such that the sum of the edge resistances along the ray is finite) starting from \ca{a neighbo\ct{u}r} of $u$ in $T$.
First, Rayleigh's monotonicity principle implies that if $v\in V(u)$ then every other vertex on the unique path from $u$ to $v$ in $T$ is also in $V(u)$.
Second, if there exist $v_1, v_2 \in V(u)$ which do not lie on a simple path from $u$ in $T$, then the component of $u$ in $T \setminus v_1$ and the component of $u$ in $T\setminus v_2$ are both recurrent and have all of $T$ as their union implying that $T$ is recurrent.

Thus, if $V(u)$ is not empty, it must be a finite path or ray in $T$ starting at $u$.
Let us rule out the case that $V(u)$ is a recurrent infinite ray. Assume that $V(u)$ is infinite and denote this ray by $(v_0,v_1,v_2,\ldots)$ with $v_0=u$. For each integer $n\geq 0$ the component of $T \setminus v_{n+1}$ containing $v_n$ is recurrent and so, by Rayleigh's monotonicity principle, the components of $T \setminus v_n$ that do not contain $v_{n+1}$ are all recurrent. Thus $T$ is decomposed to the union of the ray $(v_0,v_1, \ldots)$ and a collection of recurrent branches hanging on this ray. If the ray $V(u)$ is a recurrent, we conclude that $T$ is recurrent as well.


Suppose for contradiction that with positive probability $T$ is transient but  there exists an edge $e$ such that the component of $\rho$ in $T \setminus e$ is infinite and recurrent. Take $\eps>0$ sufficiently small \ca{so} that this edge $e$ may be taken to have $c(e)\geq\eps$ with positive probability. Denote the event that such an edge exists by $\B_\eps$. We will show that this contradicts the Mass-Transport Principle by exhibiting a mass transport such that every vertex sends a mass of at most one but some vertices receive infinite mass on the event $\B_\eps$.

From each vertex $u$ such that $V(u)$ is finite and non-empty, send mass one to the vertex $v$ in $V(u)$ that is farthest from $u$ in $T$. From each vertex $u$ such that $V(u)$ is a transient ray, send mass one to the end-point $v=e^-$ of the last edge $e$ in the path from $u$ spanned by $V(u)$ such that
 $c(e)\geq \eps$.
If $V(u)$ is empty, $u$ sends no mass.
%
  Clearly every vertex sends a total mass of at most one. However, on the event $\sB_\eps$, the vertex  $v$ that $\rho$ sends mass to receives infinite mass. Indeed, every vertex in the infinite recurrent component of $T \setminus v$ containing $\rho$ sends mass one to $v$, contradicting the
 Mass-Transport Principle.
\end{proof}

\begin{lemma}\label{coexistimpliesbranch} Let $G$ be an infinite network and let $\F$ be a sample of $\FUSF_G$. If with positive probability $\F$ has \ca{a} recurrent component and a transient component with a non-empty core, then with positive probability $\F$ has a component that is a transient tree with a recurrent branch. This holds both when edges of the trees are given the conductances inherited from $G$ and when they are given unit conductances.
\end{lemma}
\begin{proof}
We consider the case that the edges of the trees are given the conductances inherited from $G$, the other case is similar.
If two such components exist, then one can find a finite path starting at \ca{a} vertex of a recurrent component $T$ and ending in a vertex of the core of \ca{a} transient component $T'$. Moreover, \ca{by taking the shortest such path,} the starting vertex is the only vertex in $T$ and the end vertex is the only vertex in $\core(T')$.


Thus, there exists a non-random \ca{finite} simple path  $\gamma=\langle\gamma_i\rangle_{i=0}^n$ in $G$ such that the following event, denoted $\B(\gamma)$, holds with positive probability:
\begin{itemize}\itemsep0em
\item $T_\F(\gamma_0)$ is recurrent,
\item $T_\F(\gamma_i) \neq T_\F(\gamma_0)$ for $0<i\leq n$,
\item $T_\F(\gamma_n)$ is transient and
\item $\gamma_n \in \core(T_\F(\gamma_n))$ and it is the only such vertex in $\gamma$.
\end{itemize}
For each $1\leq i \leq n$, let $e_i$ be an oriented edge of $G$ with $e_i^-=\gamma_i$ and $e_i^+=\gamma_{i-1}$. Define the forests $\langle\F_i\rangle_{i=0}^n$ by setting $\F_0 = \F$ and recursively,
$$ \F_i = U(\F_{i-1}, e_i) \, , \qquad i=1,\ldots, n \, .$$
We claim that on the event $\B(\gamma)$, at least one of the two forests $\F_0$ or $\F_n$ contains a transient tree with a recurrent branch. If $\F_0$ contains such a tree we are done, so suppose not. We claim that in this case $T_{\F_n}(\gamma_n)$ is a transient tree with a recurrent branch.
Indeed, at each step of the process we are add the edge $e_i$ and remove some other edge adjacent to $\gamma_i$ in $T_{\F_{i-1}}(\gamma_i)$, so that $T_{\F_n}(\gamma_n)$ contains the tree $T_{\F_0}(\gamma_0)$ (since $\gamma_i \not \in T_{\F_0}(\gamma_0)$ for $i\geq1$) and the path $e_1, \ldots, e_n$. Moreover, since $\gamma_i \not \in \core(T_{\F_0}(\gamma_n))$
 for all $0 \leq i < n$, the tree $T_{\F_n}(\gamma_n)$ contains a branch of $T_{\F_0}(\gamma_n)$ and is therefore transient by our assumption. Thus, removing $\gamma_1$ from the transient tree $T_{\F_n}(\gamma_n)$ yields the recurrent branch $T_{\F_0}(\gamma_0)$ as required.

Denote by $\sE$ the set of subgraphs of $G$ that are transient trees with a recurrent branch. We have shown that $\P(\F_0 \in \sE) + \P(\F_n \in \sE) >0$, while by update-tolerance (\cref{updatetol})
$$\P( \F_0 \in \sE) \geq \Big ( \prod_{i=1}^n {c(e_i) \over c(e_i^-)} \Big ) \P(\F_n \in \sE) \, $$
so that  $\P(\F_0\in \sE) >0$ as claimed.
\end{proof}
 %
 \begin{figure}
 \centering
 \includegraphics[width=0.8\textwidth]{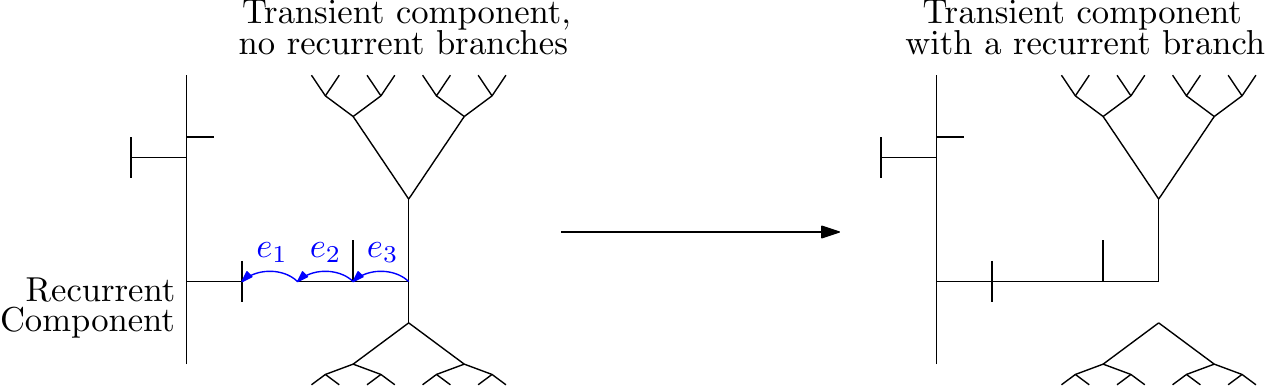}\hspace{2em}
 \caption{When recurrent components and transient components with non-empty cores and no recurrent branches coexist, a finite sequence of updates can create a transient component with a recurrent branch.}
\end{figure}

%
%
%

\begin{proof}[Proof of \cref{T:everyFUSFtreeistransient}.]
We may assume that $(G,\rho)$ is ergodic, \ca{see \cref{S:ergodicity}}.
We apply \cite[Proposition 4.9 and Theorem 6.2]{AL07} to deduce that whenever $(T,\rho)$ is an infinite unimodular random rooted (unweighed) tree with $\E[c(\rho)]<\infty$, the event that $T$ is infinitely-ended and the event that $T$ is transient coincide up to a null set and, moreover, $T$ has positive probability to be transient and infinitely-ended if and only if $\E[\deg_T(\rho)]>2$.
The expected degree of the WUSF is 2 in any unimodular random rooted network, and since the $\FUSF_G$ stochastically dominates $\WUSF_G$, the  assumption that $\FUSF_G \neq \WUSF_G$ implies that $\E [ \deg_\F (\rho) ] > 2$. Let $M>0$ and let $\F'$ be the forest obtained by deleting from $\F$ every edge $e$ such that $\max(\deg_\F(e^-),\deg_\F(e^+))\geq M$. If $M$ is sufficiently large then $\E[\deg_{\F'}(\rho)]>2$ by the monotone convergence theorem. It follows by the above that $T_{\F'}(\rho)$ is infinitely-ended and and transient (when given unit conductances) with positive probability, and consequently that the same holds for $T_\F(\rho)$ by Rayleigh monotonicity. Ergodicity of $(G,\rho,\F)$ then implies that the forest $\F$ contains a component that is infinitely-ended and transient (when given unit conductances) a.s.

Assume for contradiction that with positive probability $\F$ has a component that is finitely-ended, or equivalently a component that is recurrent when given unit conductances.  \cref{coexistimpliesbranch} then implies that with positive probability $\F$ has a transient component with a recurrent branch (when all components are given unit conductances), contradicting \cref{nobranches}.

Thus, we have that all components of $\F$ are a.s.~infinitely-ended and are transient when given unit conductances. It follows from \cite[Proposition 4.10]{AL07} that every component is also a.s.~transient when given the conductances inherited from~$G$. \qedhere
\end{proof}


\subsection{Pivotal edges for the FUSF}
%
%
%

Let $G$ be a network, let $\F$ be a sample of $\FUSF_G$ and let $\sA$ be a component property.
We say that an oriented edge $e$ of $G$ is a \textbf{$\delta$-additive pivotal} for {a vertex $v$} if
\begin{enumerate}
\item $e^+ \in T_{\F}({v})$ and $e^{-} \not \in T_\F({v})$ and,
\item given $\F$, the components  $T_{U(\F,e)}({v})$ and $T_{\F}({v})$ have different types with probability at least $\delta$.
\end{enumerate}
We say that an oriented edge $e$ is a \textbf{$\delta$-subtractive pivotal} for {$v$} if
\begin{enumerate}
\item $e^- \in T_{\F}({v})$ and $e^{+} \not \in T_\F({v})$ and,
\item given $\F$, the components $T_{U(\F,e)}({v})$ and $T_{\F}({v})$ have different types with probability at least $\delta$.
\end{enumerate}

\noindent \ct{We emphasize that when we say ``with probability at least $\delta$'' above, this is over \ct{the} randomness of $U(\F,e)$, rather than \ct{of $\F$}.}


%

\begin{lemma}\label{pivotaltype} Let $G$ be a network and let $\F$ be a sample of $\FUSF_G$. Assume that a.s.~all the components of $\F$ are transient trees with non-empty cores and that with positive probability $\F$ has components of both types $\A$ and $\neg \A$. Then for some small $\delta>0$, with positive probability there exists a vertex $v$ and an edge $e$ such that $v\in \core(F)$ and $e$ is a $\delta$-pivotal for $v$.
\end{lemma}
\begin{proof}
We argue similarly to \cref{coexistimpliesbranch}. Due to the assumptions of this lemma, there must be a component $T$ of type $\A$ and a component $T'$ of type $\neg \A$ and an edge $e \not \in \F$ connecting them. So we may form a path starting with $e$ that ends in a core vertex of $T'$ such that all edges of the path except for $e$ are in $T'$, and the last vertex of the path is the only vertex in $\core(T')$.


Hence, there exists a non-random simple path  $\gamma=\langle\gamma_i\rangle_{i=0}^n$ in $G$ such that the following event, denoted $\B(\gamma)$, holds with positive probability:
\begin{itemize}\itemsep0em
\item $T_\F(\gamma_0)$ has type $\A$,
\item $T_\F(\gamma_i) = T_\F(\gamma_j) \neq T_\F(\gamma_0)$ for all $0<i\leq j\leq n$,
\item $T_\F(\gamma_n)$ has type $\neg \A$ and,
\item $\gamma_n \in \core(T_\F(\gamma_n))$ and it is the only such vertex in $\gamma$.
\end{itemize}
 %
%
%
%

\noindent For each $1\leq i \leq n$, let $e_i$ be an oriented edge of $G$ with $e_i^-=\gamma_i$ and $e_i^+=\gamma_{i-1}$. Define the forests $\langle\F_i\rangle_{i=0}^n$ by setting $\F_0 = \F$ and, recursively,
$$ \F_i = U(\F_{i-1}, e_i) \, , \qquad i=1,\ldots, n .$$


We claim that given $\sB(\gamma)$ there exists some small $\delta>0$ such that either one of the edges $e_i$ is a $\delta$-additive pivotal for $\gamma_0$ in the forest $\F_{i-1}$, or one of the edges $e_i$ is a $\delta$-subtractive pivotal for $\gamma_n$ in the forest $\F_{i-1}$. Indeed, if there is $1 \leq i \leq n$ such that
\[ \P \big ( T_{\F_i}(\gamma_0) \in \neg \A \mid \F_{i-1} \big ) > 0 \, ,\]
(i.e., the component of $\gamma_0$ changes type with positive probability in the transition from $\F_{i-1}$ to $\F_i$), then for the first such $i$, the edge $e_i$ is a $\delta$-additive pivotal for $\gamma_0$ in $\F_{i-1}$ (since the cluster of $\gamma_0$ only grows) where $\delta>0$ is the conditional probability above.

If this does not occur, then a.s.~$T_{\F_n}(\gamma_0)$ is of type $\A$. \ca{However,} $T_{\F_n}(\gamma_0)=T_{\F_n}(\gamma_n)$ \ca{and $T_{\F_0}(\gamma_n)$ is of type $\neg \ct{\A}$}, so there is the first $1 \leq i \leq n$ in which $T_{\F_i}(\gamma_n) \in \A$ (i.e., the component of $\gamma_n$ changes type \ca{in the transition from $\F_{i-1}$ to $\F_i$}). For this $i$ we have $e_i$ is a $\delta'$-subtractive pivotal for $\gamma_n$ in $\F_{i-1}$ with
\[ \delta' = \P \big ( T_{\F_i}(\gamma_n) \in \A \mid \F_{i-1} \big ) > 0 \, .\]


Let $\sE_\delta$ be the event that there exists a vertex $v$ such that $v \in \core(\F)$ and there exists an edge $e$ that is \ca{a $\delta$-additive pivotal or $\delta$-subtractive pivotal} for $v$ in $\F$. We proved that for some small $\delta>0$ we get $\sum_{i=0}^n \P(\F_i \in \sE_\delta) > 0$. However, by update tolerance (\cref{updatetol}) it follows that
$$\P( \F_0 \in \sE) \geq \Big ( \prod_{j=1}^i {c(e_j) \over c(e_j^-)} \Big ) \P(\F_i \in \sE) \, .$$
Hence $\P(\F_0\in \sE_\delta ) >0$ for some small $\delta>0$ as claimed.
\end{proof}
%
%
%

\subsection{Proof of Theorem \ref{mainthm} for the FUSF}
%
%
%
%
%
%

\ca{Our goal in this section is to prove the following theorem.

\begin{thm} \label{mainthmfusf} Let $(G,\rho)$ be a unimodular random network with $\E[c(\rho)]<\infty$, and let $\F$ be a sample of $\FUSF_G$. On the event that $\FUSF_G \neq \WUSF_G$\ct{,} we have that for every component property $\A$, either every connected component of $\F$ has property $\A$ or none of the connected components of $\F$ have property $\sA$ almost surely.
\end{thm}
}

We follow the strategy of Lyons and Schramm \cite{LS99} while making the  changes necessary to use update-tolerance. 



\begin{proof}[\ca{Proof of \cref{mainthmfusf}}]
 Let $(G,\rho)$ be a unimodular random rooted network with $\E[c(\rho)]<\infty$, let $\F$ be a sample of $\FUSF_G$ and
let $\A$ be a component property.
\ca{Let $\langle X_n \rangle_{ n \in \Z}$ be a bi-infinite random walk on $\F$ started at $\rho$ (that is, the concatenation of two independent random walks starting at $\rho$, as in \cref{eq:reversibleunimodularcorrespondance}).}
Conditioned on the random walk $\langle X_n \rangle_{ n \in \Z}$, let $e_n$ be oriented edges chosen uniformly and independently from the set of edges at distance at most $r$ from $X_n$ in $G$. Finally, let $\{U(\F,e) : e \in E\}$ be updates of $\F$ at each edge $e$ of $G$, sampled independently of each other and of $\langle X_n \rangle_{n \in \Z}$ and $\langle e_n \rangle_{n \in \Z}$ conditional on $(G,\rho,\F)$.
We bias by $c_\F(\rho)$ so that, by \cite[Theorem 4.1]{AL07},
\begin{equation}\label{Eq:reversibleforest1}(G,\langle X_{n} \rangle_{n\in \Z},\F) \eqd (G,\langle X_{n+k} \rangle_{n\in \Z},\F) \quad \forall\, k \in \Z. \end{equation}
Let ${\widehat\P}$ denote joint distribution of the random variables $(G,\rho)$, $\F$, $\langle X_n\rangle_{n \in \Z}$,  $\langle e_n\rangle_{n \in \Z}$, and $\{U(\F,e) :e \in E\}$ under this biasing, and let $\widehat{\P}_{(G,\rho)}$ denote the conditional distribution given $(G,\rho)$ of $\F$, $\langle X_n\rangle_{n \in \Z}$, $\langle e_n\rangle_{n \in \Z}$ and $\{U(\F,e) :e \in E\}$ under the same biasing.

 By \cref{T:everyFUSFtreeistransient}, every component of $\F$ is a.s.~transient and infinitely-ended.
By \cref{pivotaltype}, there exists $\delta>0$ such that with positive probability $\rho \in \core(\F)$ and there exists a \ca{either a $\delta$-additive or $\delta$-subtractive} pivotal edge $e$ for $\rho$ in $\F$. By decreasing $\delta$ if necessary, it follows that there exists an integer $r$ such that with positive probability $\rho \in \core(\F)$ and there exists a $\delta$-pivotal edge for $\rho$ in $\F$ such that $\rho$ and $e$ are at graph distance at most $r$ in $G$ and  $c(e)/c(e-) \geq \delta$.


Conditional on $(G,\rho)$, for each edge $e$ in $G$ and $n \in \Z$, denote by $\sE^n_e$ the event that $e_n = e$ and that the trace $\{X_{n+k}\}_{k\in \Z}$ is disjoint from the components of $\F \setminus X_n$ containing $e^-$ and $e^+$.
%
For every essential spanning forest $f$ of $G$ and $n\in \Z$, we have
$$ {\widehat\P_{(G,\rho)}} ( \sE^n_e \mid \, \F = f) = {\widehat\P_{(G,\rho)}} ( \sE^n_e \mid U(\F,e) = f) \, .$$
Thus, for every event $\B \subseteq \{0,1\}^{E(G)}$ such that $\widehat\P_{(G,\rho)}(\F\in\B) > 0$, we have that
\begin{align*} {\widehat\P_{(G,\rho)}}(\sE^n_e \cap \{\F \in \B\} )  &= {\widehat\P_{(G,\rho)}} (\sE^n_e \mid \F \in \B)  \widehat\P_{(G,\rho)}(\F \in \B) \nonumber \\
&= {\widehat\P_{(G,\rho)}} (\sE^n_e \mid \{U(\F,e) \in \B\}) \widehat\P_{(G,\rho)}(\F \in \B) \nonumber\\
&= \frac{\widehat\P_{(G,\rho)}(\F \in \B)}{\widehat\P_{(G,\rho)}(U(\F,e) \in \B)} {\widehat\P_{(G,\rho)}}(\sE^n_e \cap \{U(\F,e) \in \B\}) \nonumber\\
&= \frac{\FUSF_G[c_\F(\rho)\mathbbm{1}(\F \in \B)]}{\FUSF_G^e[c_\F(\rho)\mathbbm{1}(U(\F,e) \in \B)]} {\widehat\P_{(G,\rho)}}(\sE^n_e \cap \{U(\F,e) \in \B\}). \nonumber
\end{align*}
Observe that if $e$ does not have $\rho$ as an endpoint then, by \cref{Lem:FUSfupdatestationarity},
\begin{align*} \FUSF_G[c_\F(\rho)\mathbbm{1}(\F\in\sB)] &= \frac{1}{c(e^-)}\sum_{e'^-=e^-} c(e')\FUSF_G^{e'}[c_{U(\F,e')}(\rho)\mathbbm{1}(U(\F,e')\in\sB)]\\
&\geq \frac{c(e)}{c(e^-)} \FUSF^e_G[c_{U(\F,e)}(\rho)\mathbbm{1}(U(\F,e)\in\sB)] \\
&= \frac{c(e)}{c(e^-)} \FUSF^e_G[c_{\F}(\rho)\mathbbm{1}(U(\F,e)\in\sB)],\end{align*}
and so, for every edge $e$ of $G$ not having $\rho$ as an endpoint,
\begin{align}
{\widehat\P_{(G,\rho)}}(\sE^n_e \cap \{\F \in \B\} )\geq {c(e) \over c(e^-)} {\widehat\P_{(G,\rho)}}(\sE^n_e \cap \{U(\F,e) \in \B\}) \,\label{eq:FUSFmainproof0} .
\end{align}
Update-tolerance also implies that \eqref{eq:FUSFmainproof0} holds trivially when $\widehat\P_{(G,\rho)}(\F\in\sB)=0$.

Fix $\eps>0$, and let $R$ be sufficiently large such that there exists an event $\A'$ that is measurable with respect to $(G,\rho)$ and $\F \cap B_G(\rho,R)$ and satisfies $\widehat\P(\A \symdif \A') \leq \eps$.
Such an $\A'$ exists by the assumption that $\A$ is measurable. Define the disjoint unions
\begin{align*}
\sE ^n := \bigcup_{c(e)/(e^-) \geq \delta} \sE^n_e \, \quad \text{ and } \quad
\sE^n_R := \bigcup_{e^- \notin B_G(\rho,R) \,  , c(e)/c(e^-) \geq \delta} \sE^n_e .
\end{align*}
 Condition on $(G,\rho)$, and let
\[\B =\{ \omega \in \{0,1\}^E : (G,\rho,\omega)\in \A' \setminus \A\}.\]
Summing over \eqref{eq:FUSFmainproof0} with this $\sB$ yields that, for every $R\geq 1$,
\begin{align*} \widehat\P_{(G,\rho)}( \F \in \sB) &\geq {\widehat\P_{(G,\rho)}}( \sE^n_R \cap \{\F \in \sB\}) \\ &\geq \delta {\widehat\P_{(G,\rho)}}(\sE^n_R \cap \{U(\F,e_n) \in \sB \} ) \,
\end{align*}
and hence, taking expectations,
\[ \widehat\P((G,\rho,\F)\in\sA{'}\setminus\sA)\geq \delta {\widehat\P}\big(\sE^n_R \cap \{(G,\rho,\F)\in\sA{'}\setminus\sA)\}\big).\]
By the definition of $\A'$ we have that  \[\sE^n_R\cap\{(G,\rho,U(\F,e_n)) \in \A'\} = \sE^n_R\cap\{(G,\rho,\F) \in \A'\},\] and so
\[\widehat\P((G,\rho,\F)\in\A' \setminus \A)  \geq \delta {\widehat\P}\Big(\sE^n_R \cap \big\{(G,\rho,\F) \in \A'\big\} \cap \left\{(G,\rho,U(\F,e_n)) \in \neg \A\right\} \Big).\]

Let $\sP_n$ denote the event that
$e_n$ is either a $\delta$-additive or $\delta$-subtractive pivotal edge for $X_n$.   
On the event $\sE_R^n$, the vertices $\rho$ and $X_n$ are in the same component of $U(\F,e_n)$, so that, on the event $\sP_n \cap \sE_R^n \cap \{(G,\rho,\F)\in\sA\}$,
 we have that $(G,\rho,U(\F, e_n)) \in \neg \A$ with probability at least $\delta$. Thus,
\begin{align*} \widehat\P( (G,\rho,\F)\in\A' \setminus \A)  &\geq \delta^2 {\widehat\P}(\sE^n_R \cap \{(G,\rho,\F)\in \A'\} \cap \sP_{n} ) \\ &\geq \delta^2 {\widehat\P}(\sE^n_R \cap \{(G,\rho,\F)\in \A\}\cap \sP_{n}) - \delta^2 \eps \, ,
\end{align*}
by definition of $\A'$. Since $\langle X_n\rangle_{n\geq0}$ is transient, we can take $n$ to be sufficiently large that ${\widehat\P} ( \{(G,\rho,\F)\in\A\} \cap  \sE^n \setminus \sE^n_R ) \leq \eps$. Thus, for such $n$,
\begin{align*} \widehat\P( (G,\rho,\F) \in \A' \setminus \A) &\geq \delta^2 {\widehat\P}(\sE^n \cap \{(G,\rho,\F) \in \A\} \cap \ \sP_{n}  ) - 2 \delta^2 \eps
\end{align*}
since $\A$ is a component property. Stationarity of $\langle(G,\F,\langle X_{n+k}\rangle_{k\in\Z})\rangle_{n\in\Z}$ implies that ${\widehat\P}(\sE^n \cap \{(G,\rho,\F) \in \A\}\cap \sP_n )$ does not depend on $n$, so that it suffices to show it is positive to obtain a contradiction by choosing $\eps>0$ sufficiently small.

As mentioned earlier, with positive probability $\rho \in \core(\F)$ and there exists a $\delta$-pivotal edge $e$ in $\F$ at distance at most $r$ from $\rho$ such that $c(e)/(e^-) \geq \delta$. Hence, either \[{\widehat\P}(\{\rho \in \core(\F)\}\cap\{(G,\rho,\F)\in \A\}\cap \sP_0) >0\] or \[{\widehat\P}(\{\rho \in \core(\F)\}\cap \{(G,\rho,\F) \in \neg\A\} \cap \sP_0) >0.\] Since $\neg \A$ is also a component property, we may assume without loss of generality that the former occurs.
Conditioned on the events $\{\rho \in \core(\F)\}, \{(G,\rho,\F)\in\A\}$ and $\sP_0$, the event $\sE^n$ is the event that two independent random walks from $\rho$ stay within the components of $\F \setminus \rho$ that do not contain $e_0^-$ or $e_0^+$. This occurs with positive probability since every infinite component of $T_\F(\rho)\setminus \rho$ is transient by \cref{T:everyFUSFtreeistransient} and \cref{nobranches}, concluding the proof.  \end{proof}

\section{The FUSF is either connected or has infinitely many components}\label{S:1infty}

\ca{Our goal in this section is to prove Theorems \ref{1inftycomponents} and \ref{T:Longrangedisorder}.} Let $(G,\rho)$ be a unimodular random rooted network with $\E[c(\rho)]<\infty$ and let $\F$ be a sample of $\FUSF_G$. We may assume that $(G,\rho)$ is ergodic, otherwise we take an ergodic decomposition. We may also assume that $\FUSF_G \neq \WUSF_G$ a.s., since otherwise the result follows from \cite{BLPS}.

The following is an adaptation of~\cite[Lemma~4.2]{LS99} \ca{from the unimodular transitive graph setting} \ct{to our setting.}
\ct{We omit the proof, which is very similar to that of~\cite{LS99}.}

\begin{lemma}[Component Frequencies]\label{L:frequencies} Let $(G,\rho)$ be an ergodic unimodular random rooted network with $\E[c(\rho)]<\infty$ and let $\F$ be a sample of $\FUSF_G$. Conditional on $(G,\rho)$, let $P_v$ denote the law of a random walk $\langle X_n \rangle_{n\geq0}$ on $G$ started at $v$ for each vertex $v$ of $G$, independent of $\F$.
Then there exists a measurable function $\mathrm{Freq}:\cG_\bullet^{\{0,1\}}\to[0,1]$
such that for every vertex $v$ of $G$ and every component $T$ of~$\F$,
\[ \lim_N \frac{1}{N}\sum_{n=1}^N\mathbbm{1}(X_n \in T) = \mathrm{Freq}(G,\rho,T) \quad \text{$P_v$-a.s.}\]
\end{lemma}
For each subset $W$ of $V$, we refer to $\mathrm{Freq}(W)=\mathrm{Freq}(G,\rho,W)$ as the \textbf{frequency} of $W$.
\newcommand{\commenitout}{
\begin{proof} We bias by $c(\rho)$ and work in the reversible setting.
Let $\langle X_k \rangle_{n\in\Z}$ be a bi-infinite simple random walk started from the root and for each $m<n \in \Z$ let
\[\alpha_m^n(W) = \frac{1}{m-n}\sum_{k=n}^m\mathbbm{1}(X_k \in W)\]
for each set $W\subseteq V(G)$. For each $\alpha\in[0,1]$, let
\[ \mathscr{L}_\alpha = \left\{W\subset V(G) : \lim_{n\to\infty}\alpha_0^n(W)=\alpha \quad P_\rho\text{-a.s.} \right\}\]
and let $\mathscr{L}=\bigcup_{\alpha}\mathscr{L}_\alpha$. We define Freq$(W)=\alpha$ if $W\in\mathscr{L}_\alpha$ and (arbitrarily) define Freq$(W)=0$ for $W \notin \mathscr{L}$. It is easy to see that if $W\in \mathscr{L}$ and $\langle Y_n \rangle_{n\geq 0}$ is a simple random walk started at any vertex $v$ of $G$, then
\[ \lim_N \frac{1}{N}\sum_{n=1}^N\mathbbm{1}(Y_n \in W) = \mathrm{Freq}(W)\]
and so it suffices to verify that every component $T$ of $\F$ is in $\mathscr{L}$ a.s.

\medskip

We first prove that the limit $\lim_{n\to\infty}\alpha_0^n(T)$ exists for every component $T$ of $\F$ a.s. To show this, we prove that the sequences $\alpha_0^n(T)$ are uniformly Cauchy a.s., that is,
\begin{equation}\label{Eq:Cauchy} \lim_{n\to\infty} \text{max}\left\{|\alpha_0^m(T)-\alpha_0^{m'}(T)| : m,m'\geq n,\, T \text{ a component of } \F\right\}=0.\end{equation}
Let $\eps>0$. For each $j$, let $F^j(\langle X_k \rangle_{k=n}^m)$ denote the total number of times that the $j$ components of $\F$ that are visited most often during the time interval $[n,m]$ are visited during that time interval. That is,
\[F^j(\langle X_k \rangle_{k=n}^m) =n\text{max}\left\{ \sum_{i=1}^j\alpha_0^{n-1}(T_i)  : T_1,\ldots,T_j \text{ are distinct components of $\F$}\right\}. \]
For each fixed $j$ and every $m,n$ we have that
\[F^j(\langle X_k \rangle_{k=0}^{m+n}) \leq F^j(\langle X_k \rangle_{k=0}^{n}) + F^j(\langle X_k \rangle_{k=m}^{m+n})   \]
and so the subadditive ergodic theorem implies that
$\lim_{n\to\infty}\frac{1}{n}F^j(\langle X_k \rangle_{k=0}^n)$
exists for every $j$ a.s. Let \[\alpha(j) = \lim_{n\to\infty}\left(\frac{1}{n}F^j(\langle X_k \rangle_{k=0}^n) - \frac{1}{n}F^{j-1}(\langle X_k \rangle_{k=0}^n)\right) \]
for each $j\geq 1$, so that $\alpha(j)$ is decreasing in $j$ and $\sum_{j\geq 1} \alpha(j)\leq 1$. Let $j_1$ be sufficiently large that $\alpha(j)\leq \eps/9$ for all $j\geq j_1$ and let $n_1$ be sufficiently large that
\[\left|\frac{1}{n}F^j(\langle X_k \rangle_{k=0}^n)-\frac{1}{n}F^{j-1}(\langle X_k \rangle_{k=0}^n) -\alpha(j)\right|\leq \frac{1}{9j_1+9}\eps\]
for all $j\leq j_1$ and $n\geq n_1$. Let
\[U = \{x\in[0,1] : \exists\, j\leq j_1 \text{ s.t. } |x-\alpha(j)|<\eps/(9j_1+9)\}\cup[0,\eps/3], \]
and for each $\delta>0$ let $U(\delta)$ denote the set of points at distance at most $\delta$ from $U$. We claim that for every component $T$ and every $n_2\geq n_1$, the set $\{\alpha_0^n(T):n\geq n_2\}$ is contained in some connected component of $U(1/n_2)$. Indeed, let for every component $T$ that is visited by the random walk, and all $n=1,2,\ldots$, $T$ is the $j$th most visited component at time $n$ for some $j$, and consequently there exists some $j$ depending on $n$ such that $\alpha_0^n(T)=\frac{1}{n}F^j(\langle X_k \rangle_{k=0}^n)-\frac{1}{n}F^{j-1}(\langle X_k \rangle_{k=0}^n)$. If $j\leq j_1$ and $n\geq n_1$, it follows that $\alpha_0^n(T)\in U$. If $j\geq j_1$, then the $j_1$th most visited component during the time interval $[0,n]$, say, $T'$, satisfies
\[\alpha_0^n(T)\leq\alpha_0^n(T')\leq \alpha(j_1) +\eps/9 \leq \eps/3.\]
and so $\alpha_0^n(T)\in U$ in this case also.
It follows that $\alpha_0^n(T)\in U$ for all components $T$ and all $n\geq n_1$.
Since $|\alpha_0^{n}(T)-\alpha_0^{n+1}(T)|\leq 1/(n+1)$ for all $n$, we have that the set $\{\alpha_0^n(T):n\geq n_2\}$ is contained in some connected component of $U(1/n_2)$ as claimed. When $n_2>\text{max}\{n_1,9(j_1+1)\}$ the total length of $U(1/n_2)$ is at most $\eps$, which implies that the diameter of each connected component is also at most $\eps$. This proves the claimed uniform Cauchy property \eqref{Eq:Cauchy}.

It remains to prove that the limit $\lim_{n\to\infty}\alpha_0^n(T)$ is an a.s.~constant for each component $T$ conditional on $(G,\rho)$ and $\F$. By reversibility,
\begin{align*}2\text{max}\{|\alpha_0^{2n}(T)-\alpha_0^n(T)|\,: T &\text{ a }\text{component of }\F\}\\
=&\text{max}\{|\alpha_n^{2n}(T)-\alpha_0^n(T)|\,:T\text{ a component of }\F\}\\
 \eqd &\text{max}\{|\alpha_0^{n}(T)-\alpha_{-n}^0(T)|\,:T\text{ a component of }\F\}   \end{align*}
It follows from \eqref{Eq:Cauchy} that
\[\lim_{n\to\infty}\alpha_0^n(T) = \lim_{n\to\infty}\alpha^0_{-n}(T) \text{ a.s.}\]
But $\alpha_0^n(T)$ and $\alpha^0_{-n}(T)$ are independent for each $T$ conditional on $(G,\rho)$ and $\F$, so that the limit
\[ \lim_N \frac{1}{N}\sum_{n=1}^N\mathbbm{1}(X_n \in T) \]
is an almost sure constant given $(G,\rho)$ and $\F$ as claimed.
\end{proof}
}



\begin{lemma}\label{L:nozerofrequencybranches} \ca{Let $(G,\rho)$ be a\ct{n ergodic} unimodular random rooted network with $\E[c(\rho)]<\infty$ such that $\FUSF_G \neq \WUSF_G$ a.s.~and let $\F$ be a sample of $\FUSF_G$. Conditioned on $\F$ let $P_x$ denote the law of a random walk $\langle X_n \rangle_{n\geq0}$ on $G$ started at $x$, independent of $\F$. Assume that with positive probability there exist a component of $\F$ with positive frequency. Then on this event, we a.s.~have that for every vertex $u$ of $G$ such that $\mathrm{Freq}(T_\F(u))>0$ and every edge $e \in T_\F(u)$ such that $\F\setminus e$ consists of two infinite components $K_1$ and $K_2$,}
\[P_x\Bigg(\limsup_{N\to\infty} \frac{1}{N}\sum_{\ca{n=1}}^{\ca{N}} \mathbbm{1}(X_n\in K_i\ca{)}>0 \,\Bigg|\, (G,\rho,\F)\Bigg)>0.\]
for both $i=1,2$ and every vertex $x\in G$.
\end{lemma}
\begin{proof}
We argue similarly to \cref{nobranches}.
For each vertex $u$ of $G$ and each edge $e$ in $T_\F(u)$, let $K_u(e)$ denote the component of $T_\F(u)\setminus e$ containing $u$. For every vertex $u$ such that Freq$(T_\F(u))>0$, let $E(u)$ be the set of edges $e$ in $T_\F(u)$ such that
\[P_x\Bigg(\limsup_{N\to\infty} \frac{1}{N}\sum_{\ca{n=1}}^{\ca{N}} \mathbbm{1}(X_n\in K_u(e))>0 \,\Bigg|\, (G,\rho,\F) \Bigg)=0.\]
for some  vertex $x$ (and hence every vertex).

Similarly to the proof of \cref{nobranches}, we have that for every vertex $u$, if $E(u)$ is non-empty then it is either a finite simple path or a ray in $T_\F(u)$ starting at $u$. Indeed, if $e \in E(u)$ then every edge on the path between $u$ and $e$ is also in $E(u)$, while if $e_1$ and $e_2$ do not lie on a simple path from $u$ in $T_\F(u)$ then the union of $K_u(e_1)$ and $K_u(e_2)$ is all of $T_\F(u)$ hence
\[ \limsup_{N\to\infty} \frac{1}{N}\sum_{n=1}^N \mathbbm{1}(X_n\in K_u(e_1)) + \limsup_{N\to\infty} \frac{1}{N}\sum_{n=1}^N \mathbbm{1}(X_n\in K_u(e_2)) \geq \mathrm{Freq}(T_\F(u)) >0 \]
a.s.~for every starting point $x$ of the random walk $X_n$.

First suppose that with positive probability there exists a vertex  $u\in\core(F)$ such that $E(u)$ is a finite path. Define a mass transport by sending mass one from each vertex $u$ such that $E(u)$ is finite, to the endpoint of the last edge in $E(u)$; \ca{from all other vertices send no mass}. Every vertex sends a mass of at most one while, if $E(u)$ is a finite path for some $u\in\core(F)$ and $e=(e^-,e^+)$ is the last edge of this path, then $e^+$ receives mass one from every vertex in the infinite set $K_u(e)$, contradicting the Mass-Transport Principle.

\ca{Next} suppose that \ca{with positive probability} there exists a vertex $u$ such that $E(u)$ is an infinite ray \ca{emanating} from $u$. In this case, for any other vertex $u'\in T_\F(u)$ \ca{all but finitely many edges $e$ of $E(u)$ satisfy that $u' \in K_u(e)$ and it follows that $E(u')$ is also an infinite ray and} $E(u')\symdif E(u)$ is finite.
\ca{By \cref{mainthmfusf} and ergodicity of $(G,\rho,\F)$, the set $E(u)$ is therefore an infinite ray for every vertex $u$ in $G$ a.s.  Transport unit mass from each vertex $u$ to the first vertex following $u$ in the ray $E(u)$.} Then every vertex $u$ sends unit mass, and receives $\deg_\F(u)-1$ mass.
By the Mass-Transport Principle $\E[\deg_\F(\rho)]=2$, contradicting \cite[Proposition 7.1]{AL07} and the assumption that $\FUSF_G\neq\WUSF_G$ a.s.

\ca{By Theorem \ref{T:everyFUSFtreeistransient} each component of $\F$ has a non-empty core a.s.~and by the above argument $E(u)=\emptyset$ for every core vertex $u$ for which Freq$(T_\F(u))>0$, concluding our proof.} \end{proof}

\begin{proof}[Proof of \cref{1inftycomponents}] Suppose that $\F$ has some finite number $k\geq 2$ of components a.s., which we denote $T_1,\ldots T_k$. Then for every $N$ and $v$
\[ \sum_{i=1}^k\frac{1}{N}\sum_{n=1}^N\mathbbm{1}(X_n \in T_i) = 1 \]
and so $\sum_{i=1}^k \mathrm{Freq}(T_i) = 1$. The frequency of a component is a component property and so, by \cref{mainthmfusf} we must have that Freq$(T_i)=1/k$ for all $i=1,\ldots,k$.

As in \cref{nobranches}, conditional on $(G,\rho,\F)$ there exists a simple path $\langle \gamma_i \rangle_{j\geq0}^m$ in $G$ that does not depend on $\F$ such that, with positive probability
\begin{itemize}\itemsep0em
\item $T_\F(\gamma_j) \neq T_\F(\gamma_0)$ for $j=1,\ldots, m$ and
\ca{\item All the edges $(\gamma_j,\gamma_{j+1})$ for $j=1,\ldots,m-1$ belong to the same component, and}
\item $\gamma_m \in \core(T_\F(\gamma_m))$ and it is the only such vertex in $\gamma$.
\end{itemize}
Denote this event by $\B(\gamma)$.
For each $1 \leq j \leq m$, let $e_j$ be an oriented edge of $G$ with $e_j^-=\gamma_j$ and $e_j^+=\gamma_{j-1}$. Define the forests $\langle \F_j\rangle_{j=0}^m$ recursively by setting $\F_0=\F$ and
\[\F_j = U(\F_{j-1},e_j). \]
Condition on this event $\B(\gamma)$.
  The choice of the edges $\langle e_j \rangle_{j\geq 0}$ is such that
\[ \F_m = \F \cup \{e_1\} \setminus \{d\} \]
for some edge $d$, which we orient so that $d^-=\gamma_m$, that disconnects $T_\F(\gamma_m)$ and whose removal disconnects $T_\F(\gamma_m)$ into two infinite connected components. We denote the component \ca{in $\F$} containing $\gamma_m$ by $K_1$ and the component containing $d^+$ by $K_2$. As sets of vertices, we have that
\[ T_{\F_m}(\gamma_0) = T_\F(\gamma_0)\cup K_1 \quad \text{and} \quad  T_{\F_m}(d^+) = K_2. \]
Thus, by update-tolerance (\cref{updatetol}), we have a.s.~that
\begin{equation}\label{Eq:FreqK2is0} \mathrm{Freq}(T_{\F_m}(\gamma_0)) = \mathrm{Freq}(T_\F(\gamma_0)) + \mathrm{Freq}(K_1) = \frac{1}{k} \, , \end{equation}
and so Freq$(K_1) = 0$ a.s., contradicting \cref{L:nozerofrequencybranches}. \end{proof}

\begin{proof}[Proof of \cref{T:Longrangedisorder}.]
By \cref{1inftycomponents} and the assumption that $\F$ is disconnected a.s., $\F$ has infinitely many components. For every $N$ and $v$, letting $\langle X_n\rangle_{n\geq0}$ be a simple random walk independent of $\F$ started at $v$,
\[ \sum_{\text{$T_{\ca{i}}$ is a component of $\F$}}\frac{1}{N}\sum_{n=1}^N\mathbbm{1}(X_n \in T_i) = 1 \]
and so, taking the limit
\[ \sum_{\text{$T_{\ca{i}}$ a component of $\F$}} \mathrm{Freq}(T_i) \leq 1. \]
By \cref{mainthmfusf} all the component frequencies are equal and so must all equal zero. Thus, for every component $T$ of $\F$,
\[ \lim \frac{1}{N} \sum_{n=1}^N \mathbbm{1}(X_n \in T) \to 0 \quad \text{ a.s.} \, , \]
\ca{which in particular implies that for any $\eps>0$ there exists some $n$ for which $\P(X_n \in T_\F(\rho)) \leq \eps$, hence there exists a vertex $v$ with $\P(v \in T_\F(\rho)) \leq \eps$, as required.}  \end{proof}

\begin{remark} It is possible to remove the application of indistinguishability in the above proofs. If there is a unique component with non-zero frequency a.s., \cref{L:nozerofrequencybranches} allows us to perform updates to create two components of non-zero frequency, contradicting update-tolerance and ergodicity. Otherwise, consider the component of maximal frequency. The maximal component frequency is non-random in an ergodic unimodular random rooted network. However, updating allows us to attach an infinite part of another positive-frequency component to the maximal frequency component, increasing its frequency by \cref{L:nozerofrequencybranches}, contradicting update-tolerance.
\end{remark}
\section{Indistinguishability of WUSF components}\label{S:wusfindist}

\subsection{Indistinguishability of WUSF components by tail properties.}\label{S:wusftail}

Let $G$ be a transient network.
Recall that \textbf{Wilson's algorithm rooted at infinity} \cite{Wilson96,ProppWilson,BLPS} allows us to generate a sample of $\WUSF_G$ by joining together loop-erased random walks on $G$.
Let $\gamma$ be a path in $G$ that is either finite or transient, i.e.\ visits each vertex of $G$ at most finitely many times. The \textbf{loop-erasure} $\textsf{LE}( \gamma )$ of the path $\gamma$ (introduced by Lawler \cite{Lawler80}) is formed by erasing cycles from $\gamma$ chronologically as they are created. More formally, we define $\textsf{LE}(\gamma)_i  = \gamma_{t_i}$ where the times $t_i$ are defined recursively by $t_0 = 0$ and $t_i = 1+ \max \{ t \geq t_{i-1} : \gamma_t = \gamma_{t_{i-1}}\}$.

Wilson's algorithm rooted at infinity generates a sample $\F$ of $\WUSF_G$ as follows.
Let $\{v_j : j \in \mathbb{N} \}$ be an enumeration of the vertices of $G$ and define a sequence of forests $\langle \F_i \rangle_{i\geq 0}$ in $G$ as follows:
\begin{enumerate}
\item Let $\F_0=\emptyset$.
\item Given $\F_j$, start an independent random walk from $v_{j+1}$ stopped if and when it hits the set of vertices already included in $\F_j$.
\item Form the loop-erasure of this random walk path and let $\F_{j+1}$ be the union of $\F_j$ with this loop-erased path.
\item Let $\F=\bigcup_{j\geq 0} \F_j$.
\end{enumerate}
 \ca{It is a fact} that the choice of enumeration does not affect the law of $\F$. We will require the following slight variation on Wilson's algorithm rooted at infinity.

\begin{lemma}\label{L:Wilson}
 Let $W$ be a finite set of vertices in $G$, and let $\{v_j : j \in \mathbb{N}\}$ and $\{ w_j : 1\leq j \leq |W|\}$ be enumerations of $V(G)\setminus W$ and $W$ respectively. Let $\F$ be the spanning forest of $G$ generated as follows.
\begin{enumerate}
\item Let $\F'_0=\emptyset$.
\item Given $\F'_j$, start an independent random walk from $v_{j+1}$ stopped if and when it hits the set of vertices already included in $\F'_j$.
\item Form the loop-erasure of this random walk path and let $\F'_{j+1}$ be the union of $\F'_j$ with this loop-erased path.
\item Let $\F_0=\bigcup_{j\geq 0}\F'_j$.
\item Given $\F_j$, start an independent random walk from $w_{j+1}$ stopped if and when it hits the set of vertices already included in $\F_j$.
\item Let $\F=\bigcup_{j=0}^{|W|}\F_{{j}}$.
\end{enumerate}
Then $\F$ is distributed according to $\WUSF_G$.
\end{lemma}
\begin{proof}
Let $j_0=\max\{j : v_j \text{ is adjacent to $W$}\}$ and consider the enumeration  \[v_1,\ldots,v_{j_0},w_1,\ldots,w_{|W|},v_{j_0+1},\ldots\] of $V(G)$.
Let $\{X^v : v \in V(G)\}$ be a collection of independent random walks, one from each vertex $v$ of $G$. Let $\F$ be generated using the random walks
$\{X^v : v \in V(G)\}$ as above and let $\F'$ be a sample of $\WUSF_G$ generated using Wilson's algorithm rooted at infinity, using the enumeration $v_1,\ldots,v_{j_0},w_1,\ldots,w_{|W|},v_{j_0+1},\ldots$ and the same collection of random walks $\{ X^v \}$.
Then $\F'=\F$, and so $\F'$ has distribution $\WUSF_G$ as claimed.
\end{proof}

\begin{proof}[Proof of \cref{tailproperty}] If $G$ is recurrent, then its WUSF is a.s.~connected and the statement holds trivially, so let us assume that $G$ is a.s.~transient.

For each $r>0$ let $\sG_r$ be the $\sigma$-algebra generated by the random variables $(G,\rho)$ and $\F\cap B_G(\rho,r)$.
Let $\widehat G$ be the network obtained from $G$ by contracting every edge of $\F \cap B_G(\rho,r)$ and deleting every edge of $B_G(\rho,r)\setminus \F$. Conditional on $\sG_r$, the forest $\F$ is distributed as the union of $\F\cap B_G(\rho,r)$ and a sample of $\WUSF_{\widehat G}$ \ca{by the WUSF's spatial Markov property (see \cref{S:USFbackground}).}

\begin{figure}
\centering
\includegraphics[width=0.5\textwidth]{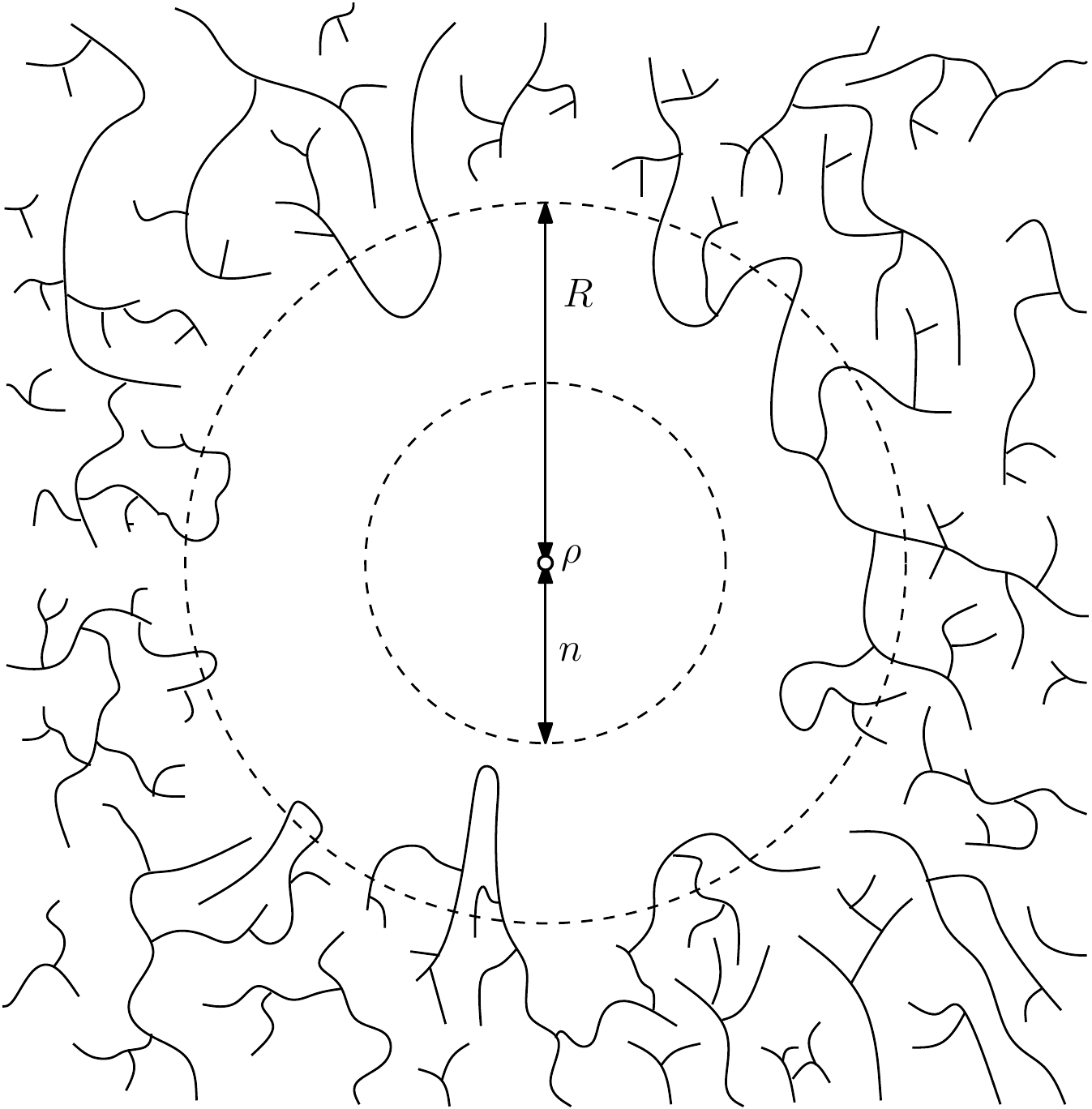}
\caption{Illustration of the forest $\F_R$ and the event $\sC_{R,n}$. The forest $\F_R$ is defined to be the union of the futures in $\F$ of each of the vertices of $G$ outside the ball of radius $R$ about $\rho$. $\sC_{R,n}$ is the event that none of these futures intersect the ball of radius $n$ about $\rho$.}
\label{fig:FR}
\end{figure}

For each integer $R\geq 0$, let $\F_R$ be the subgraph of $\F$ defined to be the union of the futures in $\F$ of every vertex in $G \setminus B_G(\rho,R)$. For each $R$, we can sample $\F_R$ conditioned on $\sG_r$ by running Wilson's algorithm on $\widehat G$ starting from every vertex $\widehat G\setminus B_G(\rho, R)$ (in arbitrary order). A vertex $v$ is contained in $\F_R$ if and only if its past $\past_\F(v)$ intersects $G \setminus B_G(\rho,R)$, and so $\bigcap_{R\geq0} \F_R=\emptyset$ a.s.~by the assumption that $\F$ has one-ended components a.s.
Conditional on $\sG_r$ and $\F_R$, the rest of $\F$ may be sampled by finishing the run of Wilson's algorithm on $\widehat G$ as described in \cref{L:Wilson}.

For each $R$ and $n$ such that $R \geq n$, let $\sC_{R,n}$ be the event that $\F_R\cap B_G(\rho,n)=\emptyset$, so that $\lim_{R \to \infty} \P(\sC_{R,n}) = 1$ for each fixed $n$ (see \cref{fig:FR}).
Let $\langle X_i \rangle_{i\geq0}$ be a random walk on $G$ started at $\rho$ and let $\langle  \widehat X_i \rangle_{i\geq 0}$ be a random walk on $\widehat G$ started at (the equivalence class in $\widehat G$ of) $\rho$. Fix $r \leq n \leq R$ and condition on $\sG_r, \F_R$ and on the event $\sC_{R,n}$.
The definition of $\F_R$ ensures that $\F \setminus \F_R$ is finite and that $T_\F(v)\setminus T_{\F_R}(v)$ is finite for every $v\in\F_R$. Thus, since $\sA$ is a tail property, the event that $T_\F(v)$ has property $\A$ is already determined by $(G,\rho)$ and $\F_R$ for every for vertex $v\in \F_R$. Thus, by \cref{L:Wilson}, the conditional probability that $\rho$ is in an $\A$ cluster equals the conditional probability that the random walk $\langle\widehat{X}_i\rangle_{i\geq0}$ first hits $\F_R$ at a vertex that belongs to an $\A$ cluster, so that
\[\P( T_\F(\rho) \in \A \mid \sG_r,\F_R, \sC_{R,n}) = \P(\langle \widehat X_i\rangle_{i\geq0} \text{ first hits $\F_R$ at an $\sA$ cluster} \mid \sG_r, \F_R, \sC_{R,n}). \]
On the event $\sC_{R,n}$ the walk $\langle \widehat X_i \rangle_{i\geq0}$ must hit $\F_R$ at time $n-r$ or greater, and so for all $j< n-r$ we have
\begin{align*}\P(T_\F(\rho) \in \A \mid \sG_r,\F_R, \sC_{R,n}) &= \P(\langle \widehat X_i\rangle_{i\geq j} \text{ first hits $\F_R$ at an $\sA$ cluster}\mid \sG_r,\F_R,\sC_{R,n})\\
&= \P(T_\F ( \widehat X_j ) \in \A \mid \sG_r,\F_R,\sC_{R,n}) \,
 \end{align*}
where the last equality is due to \cref{L:Wilson}.
That is,
\[ \P(\{T_\F(\rho) \in \A\} \cap \sC_{R,n} \mid \sG_r,\F_R) = \P(\{T_\F(\widehat X_j) \in \A\} \cap \sC_{R,n} \mid \sG_r,\F_R)\]
a.s. for all $j<n-r$.
Taking conditional expectations with respect to $\sG_r$ gives
\[ \P(\{T_\F(\rho) \in \A\} \cap \sC_{R,n} \mid \sG_r) = \P(\{T_\F(\widehat X_j) \in \A\} \cap \sC_{R,n} \mid \sG_r)\]
a.s. for all $j<n-r$. Since the events $\sC_{R,n}$ are increasing in $R$ and $\P(\sC_{R,n}) \to  1$ as $R \to \infty$, taking the limit $R\to\infty$ in the above equality gives that
\be\label{somestationarity} \P( T_\F(\rho) \in \A \mid \sG_r) = \P(T_\F ( \widehat X_j ) \in \A\mid \sG_r) \ee
for all $j < n-r$. This equality holds for all $j$ by taking $n$ to infinity.

\medskip
Let $\tau$ and $\widehat \tau$ be the last times that $\langle X_n\rangle_{n\geq0}$ and $\langle \widehat X_n \rangle_{n\geq0}$ visit $B_G(\rho,r+1)$ respectively.
Then for each vertex $v\in B_G(\rho,r+1)$, the conditional distributions
\begin{align} &\langle X_{\tau+n}\rangle_{n \geq 0}  \text{ conditioned on $X_\tau=v$ and $(G,\rho)$ and }\nonumber\\
&\langle \widehat X_{\widehat \tau +n} \rangle_{n \geq 0} \text{ conditioned on $\widehat X_{\widehat\tau}=v$ and $\sG_r$}\label{equalmeasures}  \end{align}
are equal.

Let $\mathcal{I}$ denote the invariant $\sigma$-algebra of the stationary sequence $\langle(G,X_n,\F)\rangle_{n\geq0}$. The Ergodic Theorem implies that
\[ \frac{1}{N}\sum_{i=1}^N \mathbbm{1}(T_\F(X_{i}) \in \A) \xrightarrow[N \to \infty]{\text{a.s.}} \\Y:=\P(T_\F(\rho) \in \A \mid \mathcal{I}).\]
Moreover, the random variable $Y$ is measurable with respect to the completion of the $\sigma$-algebra generated by $(G,\rho,\F)$: to see this, note that for every $a<b \in [0,1]$, Levy's 0-1 law implies that
\begin{multline*}\P\bigg(\lim_{N\to\infty}\frac{1}{N}\sum_{i=1}^{N}\mathbbm{1}(T_\F(X_{i}) \in \A) \in [a,b] \,\bigg|\, (G,\rho),\F,\langle X_n \rangle_{n=0}^k \bigg)\\
\xrightarrow[k\to\infty]{a.s.}\mathbbm{1}\bigg(\lim_{N\to\infty}\frac{1}{N}\sum_{i=1}^{N}\mathbbm{1}(T_\F(X_{i}) \in \A)\in[a,b]\bigg). \end{multline*}
But
\begin{multline*}\P\bigg(\lim_{N\to\infty}\frac{1}{N}\sum_{i=1}^{N}\mathbbm{1}(T_\F(X_{i}) \in \A) \in [a,b] \,\bigg|\, (G,\rho),\F,\langle X_n \rangle_{n=0}^k \bigg)
\\
= \P\bigg(\lim_{N\to\infty}\frac{1}{N}\sum_{i=1}^{N}\mathbbm{1}(T_\F(X_{k+i}) \in \A) \in [a,b] \,\bigg|\, (G,X_k),\F \bigg)\end{multline*}
and so, by stationarity, $\P\big(Y \!\in [a,b] \,\big|\,(G,\rho,\F)\big)\in\{0,1\}$ a.s. In particular, $Y$ is independent of $X_\tau$ given $(G,\rho)$.

Since $G$ is transient, $\tau$ is finite a.s.~and so
$$ \frac{1}{N}\sum_{i=1}^N \mathbbm{1}(T_\F(X_{\tau+i}) \in \A) \xrightarrow[N \to \infty]{\text{a.s.}} \\ Y.$$
In particular, this a.s.~convergence holds conditioned on $X_\tau = v$ for each $v$ such that $\P(X_\tau = v) > 0$. Since the support of $\widehat X_{\widehat \tau}$ is contained in the support of $X_\tau$ and the conditioned measures in \eqref{equalmeasures} are equal, we have by the above that there exists a random variable $\widehat Y$ such that
\[ \frac{1}{N}\sum_{i=1}^N \mathbbm{1}(T_\F( \widehat X_{\widehat \tau+i}) \in \A) \xrightarrow[N \to \infty]{\text{a.s.}} \widehat Y \, ,\]
and the distribution of $\widehat Y$ given $\sG_r$ and $\widehat X_{\widehat \tau}=v$ is equal to the distribution of $Y$ given $(G,\rho)$ and $X_\tau =v$, so that $\widehat Y$ is in fact independent of $\sG_r$ and $\widehat X_{\widehat \tau}$ given~$(G,\rho)$.
That is, for every $a<b\in [0,1]$,
\[\P(\widehat Y \in [a,b] \mid \sG_r, \widehat X_{\widehat \tau}=v) = \P(Y \in [a,b] \mid (G,\rho), X_\tau=v) = \P(Y \in [a,b] \mid(G,\rho))\]
so that, taking conditional expectations with respect to $(G,\rho)$,
\begin{equation}\label{eq:independence}\P(\widehat Y \in [a,b] \mid (G,\rho)) = \P(Y \in [a,b] \mid (G,\rho)) = \P(\widehat Y\in [a,b]\mid\sG_r, \widehat X_{\widehat \tau}=v) \end{equation}
establishing the independence of $\widehat Y$ from $\sG_r$ and $\widehat X_{\widehat \tau}$ conditional on $(G,\rho)$.

Since $\widehat \tau$ is finite a.s.~we also have that
\[ \frac{1}{N}\sum_{i=1}^N \mathbbm{1}(T_\F( \widehat X_{i}) \in \A) \xrightarrow[N \to \infty]{\text{a.s.}} \widehat Y. \]
Hence, by \eqref{somestationarity} and the conditional Dominated Convergence Theorem,
\begin{align*}
\P( T_\F(\rho) \in \A \mid \sG_r) &= \E\big[\mathbbm{1}(T_\F ( \widehat X_j ) \in \A)\mid \sG_r\big]\\
&= \E\!\Bigg[\frac{1}{N}\sum_{j=1}^N\mathbbm{1}(T_\F ( \widehat X_j ) \in \A) \, \Bigg| \, \sG_r\Bigg]
 \xrightarrow[N \to \infty]{a.s.} \E\!\left[\widehat Y\,\middle|\, \sG_r \right] = \E\left[\widehat Y\,\middle|\, (G,\rho)\right]. \end{align*}
It follows by similar reasoning to \eqref{eq:independence} that the event $\{T_\F(\rho) \in \sA\}$ is independent of $\F\cap B_G(\rho,r)$ conditional on $(G,\rho)$ for every $r$.
It follows that $\WUSF_G(T_\F(\rho) \in \sA) \in \{0,1\}$ a.s., and hence $\WUSF_G( T_\F(v)\in \sA) \in \{0,1\}$ for every vertex $v$ of $G$ a.s.~by stationarity. But, given $(G,\rho)$, every vertex $v$ of $G$ has positive probability of being in the same component of $\F$ as $\rho$, and so we must have that the probabilities
 \[\WUSF_G(T_\F(v)\in \sA)=\WUSF_G(T_\F(\rho) \in \sA)\in \{0,1\}\]
 agree for every vertex $v$ of $G$ a.s., so that either every component of $\F$ has type $\sA$ a.s. conditional on $(G,\rho)$ or every component of $\F$ does not have type $\sA$ a.s. conditional on $(G,\rho)$, completing the proof. \qedhere

\end{proof}


\subsection{Indistinguishability of WUSF components by non-tail properties.}
\subsubsection{Wired Cycle-Breaking for the WUSF}
Let $G$ be an infinite network and let $f$ be a spanning forest of $G$ such that every component of $f$ is infinite and one-ended. For every oriented edge $e$ of $G$, we define the \textbf{update} $U(f,e)$ of $f$ at $e$ as follows:

\begin{quote} If $e$ is a self-loop, or is already contained in $f$, let $U(f,e)=f$. Otherwise:
\begin{enumerate}\item If $e^-$ and $e^+$ are in the same component of $f$, so that $f\cup\{e\}$ contains a cycle, let $d$ be the unique edge of $f$ that is both contained in this cycle and adjacent to $e^-$ and let $U(f,e)=f\cup\{e\}\setminus\{d\}$.
\item Otherwise, let $d$ be the unique edge of $f$ such that $d$ is adjacent to $e^-$ and the component containing $e^-$ in $f\setminus\{d\}$ is finite, and let $U(f,e)=f\cup\{e\}\setminus\{d\}$.
\end{enumerate}
\end{quote}
The following are proved in \cite{H15} (in which an appropriate update rule is also developed for the case that the WUSF has multiply-ended components) \ca{and can also be proved similarly to the proof in \cref{S:cyclebreakingintheFUSF}}.
\begin{prop}\label{Prop:WUSFstationarity} Let $G$ be a network, let $\F$ be a sample of $\WUSF_G$ and suppose that every component of $\F$ is one-ended almost surely. Let $v$ be a fixed vertex of $G$ and let $E$ be an edge chosen from the set $\{e:e^-=v\}$ independently of $\F$ and with probability proportional to its conductance. Then
$U(\F,E)$ and $\F$ have the same distribution.
\end{prop}
\begin{corollary}[Update-tolerance for the WUSF]\label{updatetolwusf} Let $G$ be a network and let $\F$ be a sample of $\WUSF_G$. If every component of $\F$ is one-ended almost surely, then for every event $\sA\subset \{0,1\}^{E(G)}$ and every oriented edge $e$ in $G$,
\[ \WUSF_G(\F \in \sA) \geq \frac{c(e)}{c(e^-)}\WUSF_G(U(\F,e)\in\sA).\] \end{corollary}
\begin{proof} By \cref{Prop:WUSFstationarity},
\begin{align*} \WUSF_G(\F \in \sA) &= \frac{1}{c(e^-)}\sum_{\hat e^- = e^-}c(\hat e) \WUSF_G(U(\F,\hat e) \in \sA)\\
&\geq \frac{c(e)}{c(e^-)}\WUSF_G(U(\F,e)\in\sA).\qedhere\end{align*}
\end{proof}
\subsubsection{Pivotal edges for the WUSF}


Let $G$ be a network, and let $f$ be a spanning forest of $G$ such that every component of $f$ is infinite and one-ended and let $\sA$ be a component property. We call an oriented edge $e$ of $G$ a \textbf{good pivotal edge} for a vertex $v$ of $G$ if either
\begin{enumerate}
\item $e^+ \in T_f(v)$, $e^- \in T_f(v)$, and the type of $T_{U(f,e)}(v)$ is different from the type of $T_f(v)$ (in which case we say $e$ is a \textbf{good internal pivotal} edge for $v$),
\item $e^+ \notin T_f(v)$, $e^- \notin T_f(v)$, and the type of $T_{U(f,e)}(v)$ is different from the type of $T_f(v)$ (in which case we say $e$ is a \textbf{good external pivotal} edge for $v$),
\item $e^+ \in T_f(v)$, $e^- \notin T_f(v)$, and the type of $T_{U(f,e)}(v)$ is different from the type of $T_f(v)$ (in which case we say $e$ is a \textbf{good additive pivotal} edge for $v$), or
\item $e^- \in T_f(v)$, $e^+ \notin T_f(v)$, the component of $v$ in $T_f(v)\setminus \{e^-\}$ is infinite and the type of $T_{U(f,e)}(v)$ is different from the type of $T_f(v)$ (in which case we say $e$ is a \textbf{good subtractive pivotal} edge for $v$).
\end{enumerate}
In particular, $e$ is a good pivotal for some vertex $v$ only if infinitely many vertices change the type of their component when we update from $f$ to $U(f,e)$.

\begin{lemma}\label{specialpivotalswusf} Let $G$ be a network{,} let $\F$ be a sample of $\WUSF_G$ {and let $\sA$ be a component property}. If every component of $\F$ is one-ended a.s., then either
\begin{quote} there exists a good pivotal edge for some vertex $v$ with positive probability \end{quote}
or
\begin{quote} $\sA$ is $\WUSF_G$-equivalent to a tail component property. That is, there exists a tail component property $\sA'$ such that \[\WUSF_G((G,v,\F)\in\sA\symdif\sA')=0.\]
for every vertex $v$ of $G$.\end{quote}
\end{lemma}
\begin{proof}
Suppose that no good pivotal edges exist a.s.~and let $\sA'\subseteq \cG_\bullet^{(0,\infty)\times\{0,1\}}$ be the component property
\[ \sA'=\left\{(G,v,\omega):\;\text{\begin{minipage}{0.6\textwidth}
There exists a vertex $u \in K_\omega(v)$ and
 a one-ended essential spanning forest $f$ of $G$ such that $(G,u,f)\in~\sA$ and the symmetric differences $\omega\symdif f$ and $K_\omega(u)\symdif K_f(u)$ are finite.
\end{minipage}}\right\} \, .\]
Note that by definition $\sA'$ is a tail component property. Our goal is to show that $\sA$ and $\sA'$ have $\WUSF_G((G,v,\F)\in\sA \symdif\sA')=0$ for every vertex $v$ of $G$. One part of this assertion is easy, indeed, let $\Omega_0 \subset \cG_\bullet^{(0,\infty)\times\{0,1\}}$ be the event $\Omega_0 = \{ (G,v,\omega) : \omega \textrm{ is a one-ended essential spanning forest} \}$ so that $\WUSF_G(\Omega_0)=1$ by assumption. Then $\sA \cap \Omega_0 \subset \sA'$ since one can take $f=\omega$ and $u=v$ in the definition of $\sA'$.

The second part of this assertion is slightly more difficult and requires the use of update-tolerance. Given a one-ended essential spanning forest $\F$ and a finite sequence of oriented edges $\langle e_i \rangle_{i=1}^n$ of $G$ we define $U(\F; e_1, \ldots, e_n)$ recursively by $U(\F; e_1) = U(\F, e_1)$ and
\[ U(\F; e_1, \ldots, e_n ) = U( U(\F; e_1,\ldots, e_{n-1}), e_n ){.} \]
Let $\F$ be a sample of $\WUSF_G$ and let $\Omega_1$ be the event that for any finite sequence of edges $\langle e_i \rangle_{i=1}^n$ the forest $U(\F; e_1, \ldots, e_n)$ has no good pivotal edges.
B\ca{y} update-tolerance and the assumption that $\F$ has no good pivotal edges a.s., $\WUSF_G(\Omega_1)=1$.
Thus, it suffices to show that $\sA' \cap \Omega_0 \cap \Omega_1 \subset \sA$.


Let $(G,v,\F) \in \sA' \cap \Omega_0 \cap \Omega_1$ and let $f$ be a one-ended essential spanning forest such that $\F\symdif f$  and $T_\F(u) \symdif T_f(u)$ are finite and $(G,u,f) \in \A$ for some vertex $u\in T_\F(v)$. We will prove by induction on $|f \setminus \F|$ that there exists a vertex $u'\in T_\F(v)$ with $(G,u\ca{'},f)\in \A$ and a finite sequence of oriented edges $\langle e_i \rangle_{i=1}^n$ of $G$ such that $U(\F; e_1, \ldots, e_n) = f$ and $T_\F(u')$ and $T_{U(\F; e_1, \ldots, e_i)}(u')$  have the same type for every $1 \leq i \leq n$. Since $(G,u',f)\in\sA$ by assumption, this will imply that $(G,v,\F) \in \sA$ as desired.

To initialize the induction assume that $|f \setminus \F|=0$. Then $f \subset \F$ and, since both $\F$ and $f$ are one-ended essential spanning forests, we must have that $\F \subset f$ since any addition of an edge to $f$ creates either a cycle or a two-ended component, so that $\F=f$ and the claim is trivial.

 Next, assume that $|f\setminus \F|>0$ and let $h \in f \setminus \F$. Since $\F$ is a one-ended essential spanning forest, $\F \cup \{h\}$ contains either a cycle or a two-ended component and we can therefore find an edge $g\in \F \setminus f$ such that $\F'=\F \cup \{h\} \setminus \{g\}$ is a one-ended essential spanning forest.
The choice of $g$ is not unique, and will be important in the final case below.

First suppose $\F \cup \{h\}$ contains a cycle. In this case the choice of an edge $g$ as above is not important. The edge $g$ must be contained in this cycle since otherwise the cycle would be contained in $\F'$. Let $e_1, \ldots, e_k$ be an oriented simple path on this cycle so that $e_1=g$ and $e_k=h$. We have that $\F' = U(\F; e_k, e_{k-1},\ldots, e_2)$ by definition of the update operation.
Since none of the forests $U(\F; e_k, \ldots, e_i)$ have any good internal or external pivotal edges, we have that $T_{\F'}(u)$ has the same type as $T_\F(u)$. Lastly, $(G,u,\F)\in\Omega_{\ca{0}}\cap\Omega_{\ca{1}}$ and $|\F' \cap f| < |\F \cap f|$, so that our induction hypothesis provides us with a vertex $u'\in T_{\F'}(u)$ and a sequence of edges $e'_1,\ldots,e'_m$ such that $U(\F; e'_1, \ldots, e'_m) = f$ and $T_\F(u')$ and $T_{U(\F; e_1, \ldots, e_i)}(u')$  have the same type for every $1 \leq i \leq n$. Since this also holds when $u'$ is replaced by any vertex $u''$ in the future of $u'$ in $\F'$, we may take $u''$ such that the above hold and $u''\in T_\F(v)$.
We conclude the induction step by concatenating the two sequences $e'_1,\ldots,e'_m,e_k,\ldots,e_2$.

Now suppose that $\F \cup \{h\}$ contains a two-ended component. Let us first consider the easier case in which $u$ is not contained in this two-ended component, which is the case if and only if neither of the endpoints of $h$ are in $T_\F({u})$. In this case the choice of an edge $g$ as above is not important.
The edge $g$ must be such that the removal of $g$  disconnects the component of $\F\cup\{h\}$ containing $g$ into two infinite connected components.
 We orient $h$ so that its tail is in the component of $g$ in $\F$ and orient $g$ so that its head is in the component of $\F\setminus\{g\}$ containing $h^-$. We then take an oriented simple path in $\F$ from $g^+$ to $h^-$ and append to it the edge $h$. As above, performing the updates from the last edge of the path (that is, $h$) to the first (the edge in the path touching $g^+$) yields $\F'$. Since none of the forests $U(\F; e_k, \ldots, e_i)$ have any good external pivotal edges, we have that $T_{\F'}(u)$ has the same type as $T_\F(u)$. We may now apply our induction hypothesis to $(G,u,\F')$ as before to complete the induction step in this case.

Finally, if $\F \cup \{h\}$ contains a two-ended component and one of the endpoints of $h$ is in $T_\F(u)$. The choice of $g$ is important in this case.
Orient $h$ so that $h^+ \in T_\F(u)$ and consider the unique infinite rays from $h^+$ and $h^-$ in $\F$, denoted $e_1,e_2,\ldots$ and $e_{-1},e_{-2},\ldots$ respectively. Orient the ray $\langle e_i \rangle_{i\geq 1}$ towards infinity and the ray $\langle e_{-i} \rangle _{i\geq 0}$ towards $h^-$ so that, writing $e_0=h$, $\langle e_i \rangle_{i\in \Z}$ is an oriented bi-infinite path in $\F\cup\{h\}$.

Next consider the unique infinite ray from $u$ in $\F$. Since the symmetric difference $T_\F(u) \symdif T_f(u)$ is finite, all but finitely many of the edges in the infinite ray from $u$ in $\F$ must also be contained in the component of $u$ in $f$. Let $u'$ be the first vertex in the infinite ray from $u$ in $\F$ such that $u'$ is contained in the ray from $h^+$ in $\F$ and all of the ray from $u'$ in $\F$ is contained in $f$, so that $u'=e_k^+=e_{k+1}^-$ for some $k\geq 0$.

  Since $f$ is a one-ended essential spanning forest and contains the ray $\langle e_i\rangle_{i\geq {k+1}}$, there exists an edge $e_l$ with $l< k$ such that $e_l \notin f$. By the definition of the update operation, we have that $\F' = U(\F; -e_0, -e_1,\ldots, -e_{l-1})$ if $l>0$ and $\F'=U(\F,e_0,e_{-1},\ldots,e_{l+1})$ if $l<0$. Let $\F_j$ denote either $U(\F; -e_0, -e_1,\ldots, -e_{j})$ or $U(\F,e_0,e_{-1},\ldots,e_{-j})$ for each $j\leq l-1$ as appropriate. In either case, $u'$ is in an infinite connected component of $\F_j\setminus e_{j+1}$ for each $j$ and so, since good pivotal edges do not exist for any of the $\F_j$, the type of $T_{\F'}(u')$ is the same  as the type of $T_\F(u')$. Lastly, we also have that $(G,u',\F)\in\Omega_0\cap\Omega_1$ and $|\F' \cap f| < |\F \cap f|$, and so we may use apply our induction hypothesis to $(G,u',\F')$ as before, completing the proof. \qedhere

\end{proof}

\subsubsection{\ca{Indistinguishability of WUSF components by non-tail properties}}
\ca{
\noindent Our goal in this section is to prove the following, theorem, which in conjunction with \cref{mainthmfusf} completes the proof of \cref{mainthm}.
}
\ca{
\begin{thm} \label{mainthmwusf} Let $(G,\rho)$ be a unimodular random network with $\E[c(\rho)]<\infty$, and let $\F$ be a sample of $\WUSF_G$. Then for every component property $\A$, either every connected component of $\F$ has property $\A$ or none of the connected components of $\F$ have property $\sA$ almost surely.
\end{thm}
}
\begin{proof}[Proof of \cref{mainthmwusf}]

We may assume that $G$ is transient, since otherwise $\F$ is connected a.s.~and the claim is trivial.
Since $(G,\rho)$ becomes reversible when biased by $c(\rho)$, \cref{tailproperty}
implies that the components of $\F$ are indistinguishable by tail properties (and therefore also by properties equivalent to tail properties), so that we may assume from now on that $\sA$ is not equivalent to a tail property. In this case, \cref{specialpivotalswusf} implies that good pivotal edges exist for $\rho$ with positive probability. Without loss of generality, we may assume further that, with positive probability, $T_\F(\rho)$ has property $\sA$ and there exists a good pivotal edge for $\rho$: if not, replace $\sA$ with $\neg\sA$.
In this case, there exist a natural numbers $r$ such that, with positive probability $T_\F(\rho)$ has property $\sA$ and there exists a good pivotal edge $e$ for $\rho$ at distance at most $r$ from $\rho$ in $G$.

Let $\{\theta(e) : e \in E\}$ be i.i.d.~uniform $[0,1]$ random variables indexed by the edges of $G$, and let $\langle\omega_n\rangle_{n\geq1}$ be Bernoulli $(1-1/(n+1))$-bond percolations on $G$ defined by setting $\omega_n(e)=1$ if and only if $\theta(e)\geq 1-1/(n+1)$. By \cref{oneend}, every connected component of $\F$ is one-ended a.s.~and so every component of $\F \cap \omega_n$ is finite for every $n$ a.s.
Given $(G,\rho,\F,\theta)$, for each vertex $u$ of $G$ let $v_n(u)$ be a vertex chosen uniformly at random from the cluster of $u$ in $\F \cap \omega_n$ and let $e_n(u)$ be an oriented edge chosen uniformly from the ball of radius $r$ about $v_n(u)$ in $G$, where $(v_n(u),e_n(u))$ and $(v_{n'}(u'),e_{m'}(u'))$ are taken to be independent conditional on $(G,\rho,\F,\theta)$ if $n'\neq n$ or $u'\neq u$. We write $v_n=v_n(\rho), e_n=e_n(\rho)$ and let $\widehat\P$ denote the joint law of $(G,\rho,\F,\theta,\langle (v_n(u),e_n(u)) : u \in V \rangle_{n\geq 1})$.
The following is a special case of a standard fact about unimodular random rooted networks.
\begin{lemma}\label{samedist} $(G,\rho,v_n,\F,\theta)$ and $(G,v_n,\rho,\F,\theta)$ have the same distribution.
\end{lemma}
\begin{proof} Let $\sB\subseteq \cG^{(0,\infty)\times\{0,1\}\times[0,1]}_{\bullet\bullet}$ be an event, and for each vertex $u$ of $G$ let $K_n(u)$ by the connected component of $\omega_n\cap\F$ containing $u$. Define a mass transport by sending mass $1/|K_n(u)|$ from each vertex $u$ to every vertex $v\in K_n(u)$ such that $(G,u,v,\F,\theta)\in \sB$ (it may be that no such vertices exist, in which case $u$ sends no mass).
Then the expected mass sent by the root is
\[\widehat \E\left[\frac{1}{|K_n(\rho)|}\sum_{v\in K_n(\rho)}\mathbbm{1}((G,\rho,v,\F,\theta)\in\sB)\right]
=\widehat\P((G,\rho,v_n,\F,\theta)\in\sB)\]
while the expected mass received by the root is
\[\widehat \E\left[\frac{1}{|K_n(\rho)|}\sum_{v\in K_n(\rho)}\mathbbm{1}((G,v,\rho,\F,\theta)\in\sB)\right]
=\widehat\P((G,v_n,\rho,\F,\theta)\in\sB).\]
We conclude by applying the Mass-Transport Principle.
\end{proof}


We will also require the following simple lemma.

\begin{lemma}\label{L:choosingedgeswithpercolation}
Let $f$ be an essential spanning forest of $G$ such that every component of $f$ is one-ended.\tom{
\begin{enumerate} 
\item For every edge $e$  such that $e\notin f$ but $e^+$ and $e^-$ are in the same component of $f$, let $C(f,e)$ denote the unique cycle contained in $f\cup\{e\}$. Then for every vertex $u$ in $G$,
\begin{multline*} \widehat{\P} (e_n(u) =e \text{ and } C(f,e) \subseteq \omega_1 \mid (G,\rho),\,\F=f)\\ = \widehat{\P} (e_n(u) =e \text{ and } C(f,e) \subseteq \omega_1 \mid (G,\rho),\,\F=U(f,e)) \end{multline*}
for all $n\geq 0$.
\item For every edge $e$ of $G$, there exists $\kappa(f,e)>0$ such that for every vertex $u$ of $G$ for which at least one endpoint of $e$ is not contained in $T_f(u)$ and the component of $u$ in $f\setminus\{e^-\}$ is infinite,
\[ \widehat{\P} ( e_n(u)=e \mid (G,\rho),\, \F = f) \geq \kappa(f,e)\widehat{\P} ( e_n(u)=e \mid (G,\rho),\,\F =U(f,e))\]
for all $n\geq0$.
\end{enumerate}
}
\end{lemma}
\begin{proof}
Item \tom{$(1)$} follows immediately from the observation that, under these assumptions, the set of vertices connected to $u$ in $\omega_n\cap f$ and $\omega_n \cap U(f,e)$ are equal on the event that $C(f,e)\subseteq \omega_1$.
We \tom{now} prove item \tom{$(2)$}. If $e^+$ and $e^-$ are in the same component of $f$ \tom{or if $e^+,e^-\notin T_f(u)$} then the claim holds trivially by setting $\kappa(f,e)=1$, so suppose not. \tom{Recall that $K_{\omega_n\cap f}(u)$ is defined to be the connected component of $u$ in $\omega_n \cap f$.} Define
 \[ \kappa_1(u,f;\omega_n) = \frac{1}{|K_{\omega_n\cap f}(u)|} \]
 and \[ \kappa_2(u,f,e;\omega_n) = \sum_{\{v \in K_{\omega_n\cap f}(u) \,:\, d(v,e)\leq r\}}\frac{1}{|\{e' \in E: d(v,e')\leq r\}|}. \]
Then conditional on $(G,\rho)$, $\F=f$, and $\omega_n$, the probability that $e_n(u)=e$ for each oriented edge $e$ of $G$ equals
\[ \kappa_1(u,f;\omega_n)\kappa_2(u,f,e;\omega_n).\]
Let $W$ denote the union of the finite components of $f \setminus \{e^+,e^-\}$. Our assumptions on $e$, $u$ and $f$ imply that $T_{U(f,e)}(u) \bigtriangleup T_{f}(u)$ is contained in $W$, so that
\begin{align*} \kappa_1(u,f; \omega_n)^{-1} \;=\; |K_{\omega_n\cap f}(u)| \;\leq\; |K_{\omega_n\cap U(f,e)}(u)| + |W|
\;=\;\kappa_1(u,U(f,e);\omega_n)^{-1} \asaf{+} |W| \, , \end{align*}
and so
\[ \kappa_1(u,f;\omega_n)\geq \frac{1}{1+|W|}\kappa_1(u,U(f,e);\omega_n),\]
since $\kappa_1(u,U(f,e);\omega_n)\leq 1$. Let
\[ \kappa_2^-(e) = \min\left\{ |\{e' \in E: d(v,e')\leq r\}|^{-1} : v\in V(G), d(v,e)\leq r\right\} >0. \]
Suppose that $\kappa_2(u,U(f,e),e;\omega_n)>0$. Then there is a vertex $x$ in the tree $K_{\omega_n \cap U(f,e)}(u)$ such that $d(x,e)\leq r$ and $x$ is still connected to $u$ in $K_{\omega_n\cap U(f,e)}(u) \setminus e$. This $x$ is therefore also be connected to $u$ in $\omega_n \cap f$, and so
\begin{align*}\kappa_2(u,f,e;\omega_n) \geq |\{e' \in \cha{E}: d(x,e')\leq r\}|^{-1} \geq \kappa_2^-(e) \, , \end{align*}
and thus,
\[\kappa_2(u,f,e;\omega_n) \geq \kappa_2^-(e)\mathbbm{1}\big(\kappa_2(u,U(f,e),e;\omega_n)>0\big).\]
But $\kappa_2(u,U(f,e),e;\omega_n)$ is bounded above by
\[\kappa_2(u,U(f,e),e;\omega_n) \leq \kappa_2^+(e) :=  \sum_{\{v :\, d(v,e)\leq r\}}\frac{1}{|\{e' \in E: d(v,e')\leq r\}|} \]
and so
\[\kappa_2(u,f,e;\omega_n) \geq \frac{\kappa_2^-(e)}{\kappa_2^+(e)}\kappa_2(u,U(f,e),e;\omega_n).\]
We obtain that
\begin{equation}\label{Eq:kappa1}\kappa_1(u,f,e;\omega_n)\kappa_2(u,f,e;\omega_n) \geq \frac{\kappa_2^-(e)}{(1+|W|)\kappa_2^+(e)}\kappa_1(u,U(f,e);\omega_n)\kappa_2(u,U(f,e),e;\omega_n).\end{equation}
The claim follows by setting
\[ \kappa(f,e)=\frac{\kappa_2^-(e)}{(1+|W|)\kappa_2^+(e)}\]
and taking expectations over $\omega_n$ in \eqref{Eq:kappa1}.
\end{proof}





Given $(G,\rho,\F,\theta)$ and a positive $\delta>0$, we say that an oriented edge $e$ of $G$ is $\delta$~\textbf{-update-friendly}~if
\begin{enumerate}
\item $c(e)/c(e^-) \geq \delta$, and
\item $\kappa(\F,e)\geq \delta$, and
\item if $e\notin \F$ but $e^+$ and $e^-$ are in the same component of $\F$, then $C(\F,e) \subseteq \omega_1$.
\end{enumerate}
Note that if $e$ is $\delta$-update-friendly for $(G,\rho,\F,\theta)$ then it is also $\delta$-update-friendly for $(G,\rho,U(\F,e),\theta)$.
By assumption, there exists $\delta>0$ such that with positive probability $T_\F(\rho)$ has property $\sA$ and there exists a good pivotal edge $e$ for $\rho$ at distance at most $r$ from $\rho$ in $G$ such that $e$ is $\delta$-update-friendly.




Conditional on $(G,\rho)$,  for each edge $e$ of $G$ and $n \in \Z$, let $\sE_e^n$ denote the event that $e$ is $\delta$-update-friendly and $e_n=e$. Write $\widehat{\P}_{(G,\rho)}$ for $\widehat{\P}$ conditioned on $(G,\rho)$. Applying part (2) of \cref{L:choosingedgeswithpercolation} if $e^+,e^-$ are both in $T_\F(\rho)$ and part (1) otherwise, we deduce from the definition of $\delta$-update-friendliness that for every event $\B \in \{0,1\}^{E(G)}$ such that $\WUSF_G(\F\in\B) > 0$,
\begin{align}\label{eq:WUSFnontailproof} \widehat{\P}_{(G,\rho)}(\sE^n_e \cap \{\F \in \B\} )  &= \widehat{\P}_{(G,\rho)} (\sE^n_e \mid \F \in \B) \WUSF_G(\F \in \B) \nonumber\\
&\geq \delta \widehat{\P}_{(G,\rho)} (\sE^n_e \mid \{U(\F,e) \in \B\}) \WUSF_G(\F \in \B) \nonumber\\
&= \delta\mathbbm{1}\!\left(\frac{c(e)}{c(e^-)}\geq\delta\right)\frac{\WUSF_G(\F \in \B)}{\WUSF_G(U(\F,e) \in \B)}\nonumber\\ &\hspace{12em}\cdot \widehat{\P}_{(G,\rho)}(\sE^n_e \cap \{U(\F,e) \in \B\}) \nonumber\\
&\geq \delta^2 \widehat{\P}_{(G,\rho)}(\sE^n_e \cap \{U(\F,e) \in \B\}) \, ,
\end{align}
where the last inequality is by update-tolerance (\cref{updatetolwusf}). Update-tolerance also implies that this inequality holds trivially when $\WUSF_G(\F\in\B)=0$.

Fix $\eps>0$, and let $R$ be sufficiently large that there exists an event $\A'$ that is measurable with respect to the $\sigma$-algebra generated by $(G,\rho)$ and $\F \cap B_G(\rho,R)$ and has $\widehat\P((G,\rho,\F)\in\A \symdif \A') \leq \eps$.
Define the disjoint unions
\begin{align*}
\sE ^n := \bigcup_{c(e)/(e^-) \geq \delta} \sE^n_e \, \quad \text{ and } \quad
\sE^n_R := \bigcup_{e^- \notin B_G(\rho,R) \,  , c(e)/c(e^-) \geq \delta} \sE^n_e .
\end{align*}
 Condition on $(G,\rho)$, and let
\[\B =\{ \omega \in \{0,1\}^E : (G,\rho,\omega)\in \A' \setminus \A\}.\]
Summing over \asaf{\eqref{eq:WUSFnontailproof}} with this $\sB$ yields that
\begin{align*} \widehat\P_{(G,\rho)}( \F \in \sB) &\geq {\widehat\P_{(G,\rho)}}( \sE^n_R \cap \{\F \in \sB\}) \\ &\geq \delta^{\asaf{2}} {\widehat\P_{(G,\rho)}}(\sE^n_R \cap \{U(\F,e_n) \in \sB \} ) \,
\end{align*}
and hence, taking expectations,
\[ \widehat\P((G,\rho,\F)\in\sA\asaf{'}\setminus\sA)\geq \delta^{\asaf{2}} {\widehat\P}\big(\sE^n_R \cap \{(G,\rho,\F)\in\sA\asaf{'}\setminus\sA\}\big).\]
By the definition of $\A'$ we have that  \[\sE^n_R\cap\{(G,\rho,U(\F,e_n)) \in \A'\} = \sE^n_R\cap\{(G,\rho,\F) \in \A'\},\]
and so
\begin{multline}\label{Eq:WUSFestimate1}\widehat\P((G,\rho,\F)\in\A' \setminus \A)\\  \geq \delta^{\asaf{2}} {\widehat\P}\Big(\sE^n_R \cap \big\{(G,\rho,\F) \in \A'\big\} \cap \left\{(G,\rho,U(\F,e_n)) \in \neg \A\right\} \Big).\end{multline}

Let $\sP_n$ denote the event that
 $e_n$ is a good pivotal edge for $v_n$. We claim that if $\sP_n$ occurs and $\rho$ is not in the past of $v_n$, then $T_{U(\F,e_n)}(\rho) = T_{U(\F,e_n)}(v_n)$ and $(G,\rho,U(\F,e_n)) \in \neg\sA$. If $e_n$ is a good internal, external or additive pivotal for $v_n$, then clearly $\rho$ and $v_n$ are in the same component of $U(\F,e_n)$, and, since $e_n$ is pivotal for $v_n$ we deduce that $(G,\rho,U(\F,e_n)) \in \neg \sA$.  If $e_n$ is a good subtractive pivotal edge for $v_n$ then the component of $v_n$ in $\F \setminus \{e_n^-,e_n^+\}$ is infinite and, since $\rho$ is not in the past of $v_n$, $\rho$ and $v_n$ must be in the same component of $\F \setminus \{e_n^-,e_n^+\}$. It follows that $\rho$ and $v_n$ are in the same component of $U(\F,e_n)$, and so $U(\F,e_n) \in \neg \sA_\rho$ as before. Combining this with \eqref{Eq:WUSFestimate1}, we have
\begin{align*} \widehat\P\big( (G,\rho,\F) \in \A' \setminus \A\big)  \geq \delta^2  \widehat{\P}\big(\sE^n_R \cap \big\{(G,\rho,\F) \in \A'\big\} \cap \sP_{n}  \cap \big\{ \rho \not \in \textrm{past}_\F(v_n)\big\} \big) \, .
\end{align*}
\begin{lemma} $\widehat \P(\rho \in \mathrm{past}(v_n)) \to 0$ as $n \to \infty$. \end{lemma}
\begin{proof} By \cref{samedist},
\[\widehat \P\big(\rho \in \textrm{past}_\F(v_n)\big) = \widehat \P\big(v_n \in \textrm{past}_\F(\rho)). \]
Observe that past$(\rho)$ is finite, while the size of the component of $\rho$ in $T_\F(\rho)\cap \omega_n$ tends to infinity as $n \to \infty$. Since $v_n$ is defined to be a uniform vertex of the this component, it follows that
\[ \widehat \P(v_n \in \textrm{past}_\F(\rho) \mid (G,\rho, \F, \theta) ) \xrightarrow[n\to\infty]{a.s.} 0\]
and the claim follows by taking expectations.
\end{proof}
Thus, taking $n$ sufficiently large that $\widehat \P(\rho \in \textrm{past}_\F(v_n))<\eps$, we have that
\[ \P( (G,\rho,\F) \in \A' \setminus \A)  \geq \delta^2  \widehat{\P}(\sE^n_R \cap \{(G,\rho,\F) \in \A'\} \cap  \sP_{n} ) - \delta^2  \eps. \] By definition of $\A'$, we then have that
$$ \P((G,\rho,\F) \in \A' \setminus \A) \geq \delta^2  \widehat{\P}(\sE^n_R \cap \{(G,\rho,\F) \in \A\} \cap \sP_{n} ) - 2\delta^2  \eps \, .$$
We can further choose $n$ to be sufficiently large that $\widehat{\P}(\sE^n \setminus \sE^n_R) \leq \eps$, so that
\begin{align} \nonumber \P( (G,\rho,\F) \in \A' \setminus \A) &\geq \delta^2  \widehat{\P}(\sE^n \cap \{(G,\rho,\F) \in \A\} \cap \sP_{n} ) - 3 \delta^2  \eps \\ &= \delta^2 \widehat{\P}(\sE^n \cap \{(G,\asaf{v_n},\F)\in\A\}\cap \sP_n ) - 3 \delta^2  \eps \,\label{eq:WUSFestimatefinal}
\end{align}
where in the second equality we have used the fact that $\A$ is a component property. Observe that, by \cref{samedist}, the probability  $\widehat{\P}(\sE^n \cap \{(G,\asaf{v_n},\F)\in\A\}\cap \sP_n )>0$ does not depend on $n$.
It does not depend on $\eps$ either, and so \eqref{eq:WUSFestimatefinal} contradicts the definition of $\A'$ when $\eps$ is taken to be sufficiently small.
\end{proof}

\section*{Acknowledgements}
 We are grateful to Russ Lyons for many comments, corrections  and improvements to the manuscript, and also to Ander Holroyd and Yuval Peres for useful discussions.  TH thanks Tel Aviv University and both authors thank the Issac Newton Institute, where part of this work was carried out, for their hospitality.  This project is supported by NSERC.
 \footnotesize{
\providecommand{\bysame}{\leavevmode\hbox to3em{\hrulefill}\thinspace}
\providecommand{\MR}{\relax\ifhmode\unskip\space\fi MR }
\providecommand{\MRhref}[2]{%
  \href{http://www.ams.org/mathscinet-getitem?mr=#1}{#2}
}
\providecommand{\href}[2]{#2}

}

\begin{thebibliography}{10}

\bibitem{AizWar06}
Michael Aizenman and Simone Warzel, \emph{The canopy graph and level statistics
  for random operators on trees}, Math. Phys. Anal. Geom. \textbf{9} (2006),
  no.~4, 291--333 (2007). \MR{2329431 (2008f:60071)}

\bibitem{AL07}
David Aldous and Russell Lyons, \emph{Processes on unimodular random networks},
  Electron. J. Probab. \textbf{12} (2007), no. 54, 1454--1508. \MR{2354165
  (2008m:60012)}

\bibitem{unimodular2}
Omer Angel, Tom Hutchcroft, Asaf Nachmias, and Gourab Ray, \emph{A dichotomy
  for random planar maps}, In preparation.

\bibitem{BLPS99}
I.~Benjamini, R.~Lyons, Y.~Peres, and O.~Schramm, \emph{Group-invariant
  percolation on graphs}, Geom. Funct. Anal. \textbf{9} (1999), no.~1, 29--66.
  \MR{1675890 (99m:60149)}

\bibitem{BPP14}
I.~{Benjamini}, E.~{Paquette}, and J.~{Pfeffer}, \emph{{Anchored expansion,
  speed, and the hyperbolic Poisson Voronoi tessellation}}, ArXiv e-prints
  (2014).

\bibitem{BC2011}
Itai Benjamini and Nicolas Curien, \emph{Ergodic theory on stationary random
  graphs}, Electron. J. Probab. \textbf{17} (2012), no. 93, 20. \MR{2994841}

\bibitem{BeKePeSc04}
Itai Benjamini, Harry Kesten, Yuval Peres, and Oded Schramm, \emph{Geometry of
  the uniform spanning forest: transitions in dimensions {$4,8,12,\dots$}},
  Ann. of Math. (2) \textbf{160} (2004), no.~2, 465--491. \MR{2123930
  (2005k:60026)}

\bibitem{BLPS}
Itai Benjamini, Russell Lyons, Yuval Peres, and Oded Schramm, \emph{Uniform
  spanning forests}, Ann. Probab. \textbf{29} (2001), no.~1, 1--65. \MR{1825141
  (2003a:60015)}

\bibitem{BurPe93}
Robert Burton and Robin Pemantle, \emph{Local characteristics, entropy and
  limit theorems for spanning trees and domino tilings via
  transfer-impedances}, Ann. Probab. \textbf{21} (1993), no.~3, 1329--1371.
  \MR{1235419 (94m:60019)}

\bibitem{PSHIT}
Nicolas Curien, \emph{Planar stochastic hyperbolic infinite triangulations},
  arXiv:1401.3297, 2014.

\bibitem{EpMo09}
Inessa Epstein and Nicolas Monod, \emph{Nonunitarizable representations and
  random forests}, Int. Math. Res. Not. IMRN (2009), no.~22, 4336--4353.
  \MR{2552305 (2010j:22007)}

\bibitem{Gab05}
D.~Gaboriau, \emph{Invariant percolation and harmonic {D}irichlet functions},
  Geom. Funct. Anal. \textbf{15} (2005), no.~5, 1004--1051. \MR{2221157
  (2007m:60294)}

\bibitem{Gab10}
Damien Gaboriau, \emph{What is {$\ldots$} cost?}, Notices Amer. Math. Soc.
  \textbf{57} (2010), no.~10, 1295--1296. \MR{2761803 (2012c:37005)}

\bibitem{GabLyons07}
Damien Gaboriau and Russell Lyons, \emph{A measurable-group-theoretic solution
  to von {N}eumann's problem}, Invent. Math. \textbf{177} (2009), no.~3,
  533--540. \MR{2534099 (2011a:37010)}

\bibitem{GrimFKbook}
Geoffrey Grimmett, \emph{The random-cluster model}, Grundlehren der
  Mathematischen Wissenschaften [Fundamental Principles of Mathematical
  Sciences], vol. 333, Springer-Verlag, Berlin, 2006. \MR{2243761
  (2007m:60295)}

\bibitem{Hagg95}
Olle H{\"a}ggstr{\"o}m, \emph{Random-cluster measures and uniform spanning
  trees}, Stochastic Process. Appl. \textbf{59} (1995), no.~2, 267--275.
  \MR{1357655 (97b:60170)}

\bibitem{HebSaCo93}
W.~Hebisch and L.~Saloff-Coste, \emph{Gaussian estimates for {M}arkov chains
  and random walks on groups}, Ann. Probab. \textbf{21} (1993), no.~2,
  673--709. \MR{1217561 (94m:60144)}

\bibitem{HLMPPW}
Alexander~E. Holroyd, Lionel Levine, Karola M{\'e}sz{\'a}ros, Yuval Peres,
  James Propp, and David~B. Wilson, \emph{Chip-firing and rotor-routing on
  directed graphs}, In and out of equilibrium. 2, Progr. Probab., vol.~60,
  Birkh\"auser, Basel, 2008, pp.~331--364. \MR{2477390 (2010f:82066)}

\bibitem{H15}
Tom Hutchcroft, \emph{Wired cycle-breaking dynamics for uniform spanning
  forests}, http://arxiv.org/abs/1504.03928.

\bibitem{JarRed08}
Antal~A. J{\'a}rai and Frank Redig, \emph{Infinite volume limit of the abelian
  sandpile model in dimensions {$d\geq 3$}}, Probab. Theory Related Fields
  \textbf{141} (2008), no.~1-2, 181--212. \MR{2372969 (2009c:60268)}

\bibitem{JarWer14}
Antal~A. J{\'a}rai and Nicol{\'a}s Werning, \emph{Minimal configurations and
  sandpile measures}, J. Theoret. Probab. \textbf{27} (2014), no.~1, 153--167.
  \MR{3174221}

\bibitem{Ken00}
Richard Kenyon, \emph{The asymptotic determinant of the discrete {L}aplacian},
  Acta Math. \textbf{185} (2000), no.~2, 239--286. \MR{1819995 (2002g:82019)}

\bibitem{Kirch1847}
Gustav Kirchhoff, \emph{Ueber die {A}ufl\"osung der {G}leichungen, auf welche
  man bei der {U}ntersuchung der linearen {V}ertheilung galvanischer {S}tr\"ome
  gef\"uhrt wird.}, Ann. Phys. und Chem. (1847), no.~72, 497--508.

\bibitem{Lawler80}
Gregory~F. Lawler, \emph{A self-avoiding random walk}, Duke Math. J.
  \textbf{47} (1980), no.~3, 655--693. \MR{587173 (81j:60081)}

\bibitem{LP:book}
R.~Lyons and Y.~Peres, \emph{Probability on trees and networks}, Cambridge
  University Press, 2015, In preparation. Current version available at
  \hfil\break {\tt http://mypage.iu.edu/\string~rdlyons/}.

\bibitem{Lyons09}
Russell Lyons, \emph{Random complexes and {$l^2$}-{B}etti numbers}, J. Topol.
  Anal. \textbf{1} (2009), no.~2, 153--175. \MR{2541759 (2010k:05130)}

\bibitem{LMS08}
Russell Lyons, Benjamin~J. Morris, and Oded Schramm, \emph{Ends in uniform
  spanning forests}, Electron. J. Probab. \textbf{13} (2008), no. 58,
  1702--1725. \MR{2448128 (2010a:60031)}

\bibitem{LS99}
Russell Lyons and Oded Schramm, \emph{Indistinguishability of percolation
  clusters}, Ann. Probab. \textbf{27} (1999), no.~4, 1809--1836. \MR{1742889
  (2000m:60114)}

\bibitem{MajDhar92}
S.~N. Majumdar and D.~Dhar, \emph{Equivalence between the abelian sandpile
  model and the $q \to 0$ limit of the potts model}, Physica A (1992), no.~185,
  129--145.

\bibitem{Morris03}
Ben Morris, \emph{The components of the wired spanning forest are recurrent},
  Probab. Theory Related Fields \textbf{125} (2003), no.~2, 259--265.
  \MR{1961344 (2004a:60024)}

\bibitem{Pem91}
Robin Pemantle, \emph{Choosing a spanning tree for the integer lattice
  uniformly}, Ann. Probab. \textbf{19} (1991), no.~4, 1559--1574. \MR{1127715
  (92g:60014)}

\bibitem{ProppWilson}
James~Gary Propp and David~Bruce Wilson, \emph{How to get a perfectly random
  sample from a generic {M}arkov chain and generate a random spanning tree of a
  directed graph}, J. Algorithms \textbf{27} (1998), no.~2, 170--217, 7th
  Annual ACM-SIAM Symposium on Discrete Algorithms (Atlanta, GA, 1996).
  \MR{1622393 (99g:60116)}

\bibitem{Timar15}
\'Ad\'am Tim\'ar, \emph{Indistinguishability of components of random spanning
  forests}, http://arxiv.org/abs/1506.01370.

\bibitem{Wilson96}
David~Bruce Wilson, \emph{Generating random spanning trees more quickly than
  the cover time}, Proceedings of the {T}wenty-eighth {A}nnual {ACM}
  {S}ymposium on the {T}heory of {C}omputing ({P}hiladelphia, {PA}, 1996), ACM,
  New York, 1996, pp.~296--303. \MR{1427525}

\end{thebibliography}
\end{document}